\documentclass[10pt]{amsart}
\allowdisplaybreaks

\usepackage{amsfonts} \usepackage{amsmath} \usepackage{amssymb}
\usepackage{amsthm}
\usepackage{graphicx} \usepackage{enumerate} \usepackage{multicol}
\usepackage{mathrsfs} \usepackage{xypic}
\usepackage{enumerate}

\newcounter{dummy} \numberwithin{dummy}{section}
\newtheorem{theorem}[dummy]{Theorem}
\newtheorem{corollary}[dummy]{Corollary}
\newtheorem{lemma}[dummy]{Lemma}

\newtheorem{proposition}[dummy]{Proposition}
\theoremstyle{remark}
\newtheorem{remark}[dummy]{Remark}
\newtheorem{example}[dummy]{Example}

\DeclareFontFamily{U}{mathx}{\hyphenchar\font45}
\DeclareFontShape{U}{mathx}{m}{n}{
      <5> <6> <7> <8> <9> <10>
      <10.95> <12> <14.4> <17.28> <20.74> <24.88>
      mathx10
      }{}
\DeclareSymbolFont{mathx}{U}{mathx}{m}{n}
\DeclareFontSubstitution{U}{mathx}{m}{n}
\DeclareMathAccent{\widecheck}{0}{mathx}{"71}
\DeclareMathAccent{\wideparen}{0}{mathx}{"75}

\DeclareMathOperator{\real}{\mathbb R}
\DeclareMathOperator{\integer}{\mathbb Z}

\newcommand{\calB}{\mathcal B}
\newcommand{\calE}{\mathcal E}
\newcommand{\calF}{\mathcal F}
\newcommand{\calH}{\mathcal H}
\newcommand{\calL}{\mathcal L}
\newcommand{\calR}{\mathcal R}
\newcommand{\calV}{\mathcal V}

\newcommand{\frakg}{\mathfrak g}
\newcommand{\frakh}{\mathfrak h}
\newcommand{\frakk}{\mathfrak k}

\newcommand{\scrF}{\mathscr F}

\DeclareMathOperator{\Ann}{Ann}

\DeclareMathOperator{\Sym}{Sym}

\DeclareMathOperator{\Dom}{Dom}
\DeclareMathOperator{\id}{id}
\DeclareMathOperator{\inc}{inc}
\DeclareMathOperator{\pr}{pr}
\DeclareMathOperator{\rank}{rank}
\DeclareMathOperator{\spn}{span}
\DeclareMathOperator{\tr}{tr}

\newcommand{\prob}{\mathbb P}
\newcommand{\expect}{\mathbb E}
\newcommand{\equivM}{\stackrel{\rm loc}{=}}

\DeclareMathOperator{\ad}{ad}

\DeclareMathOperator{\tensorg}{\mathbf g}
\DeclareMathOperator{\dv}{div}

\DeclareMathOperator{\grad}{grad}
\DeclareMathOperator{\Ric}{Ric}
\DeclareMathOperator{\vol}{vol}
\DeclareMathOperator{\dvol}{\text{$d$}vol}

\newcommand{\tensorh}{\mathbf h}
\newcommand{\rnabla}{\mathring{\nabla}}
\newcommand{\znabla}{\,\mathring{\phantom{\nabla}\!\!\!\!\!\!}\mathring{\!\!\nabla}}
\newcommand{\shh}{\sharp^{{\mathbf h}^*}}
\newcommand{\srL}{\Delta_{\mathbf{h}}}
\newcommand{\roughL}{\Delta_{\mathbf{h}}^\prime}

\newcommand{\tensorv}{\mathbf v}
\newcommand{\verticalL}{\Delta_{\mathbf{v}}}
\newcommand{\shv}{\sharp^{{\mathbf v}^*}}
\newcommand{\shvk}{\sharp^{{\mathbf v}^*_k}}

\newcommand{\tensors}{\mathbf s}

\newcommand{\tensorq}{\mathbf q}
\newcommand{\metricd}{\mathsf d}
\newcommand{\ve}{\varepsilon}
\newcommand{\sfGamma}{\mathsf \Gamma}

\newcommand{\RicH}{\Ric_{\mathcal{H}}^{\mathstrut}}
\newcommand{\RicV}{\Ric_{\mathcal{V}}^{\mathstrut}}
\newcommand{\RicHV}{\Ric_{\mathcal{HV}}}
\newcommand{\rRicH}{\rho_{\mathcal{H}}^{\mathstrut}}
\newcommand{\MRicHV}{\mathscr{M}_{\mathcal{HV}}}
\newcommand{\McalR}{\mathscr{M}_{\mathcal{R}}^{\mathstrut}}
\newcommand{\mcalR}{m_{\mathcal{R}}^{\mathstrut}}
\newcommand{\mcalRvar}{m_{\mathcal{R}}}
\newcommand{\MII}{\mathscr{M}_{\rnabla \tensorv^*}}
\newcommand{\rLv}{\rho_{\roughL \! \tensorv^*}}

\newcommand\newbullet{{\kern.8pt\displaystyle\centerdot\kern.8pt}}
\newcommand\newbulletsub{{\centerdot\kern.8pt}}

\numberwithin{equation}{section}

\title[Curvature-dimension inequalities on sR-manifolds, Part~II]{Curvature-dimension inequalities on sub-Riemannian manifolds obtained from Riemannian foliations, Part~II}
\author[E. Grong, A. Thalmaier]{Erlend Grong \\ Anton Thalmaier}

\address{Mathematics Research Unit, University of Luxembourg, Luxembourg.}
\email{erlend.grong@uni.lu}

\address{Mathematics Research Unit, University of Luxembourg, Luxembourg.}
\email{anton.thalmaier@uni.lu}

\subjclass[2010]{58J35 (53C17,58J99)}

\keywords{Sub-Riemannian geometry; Diffusions of hypoelliptic operators; Curvature-dimension inequality; Riemannian foliations}

\begin{document}

\begin{abstract}
Using the curvature-dimension inequality proved in Part~I, we look
  at consequences of this inequality in terms of the interaction
  between the sub-Riemannian geometry and the heat semigroup $P_t$
  corresponding to the sub-Laplacian. We give bounds for the gradient,
  entropy, a Poincar\'e inequality and a Li-Yau type inequality. These
  results require that the gradient of $P_t f$ remains uniformly
  bounded whenever the gradient of $f$ is bounded and we give several
  sufficient conditions for this to hold.
\end{abstract}

\maketitle

\section{Introduction} \label{sec:Introduction} One of the most
important relations connecting the geometric properties of a
Riemannian manifold $(M, \tensorg)$ with the properties of its Laplace
operator $\Delta$ is the curvature-dimension inequality given by
$$\frac{1}{2} \Delta \|\grad f\|^2_{\tensorg} 
- \langle\grad f, \grad \Delta f\rangle_{\tensorg} \geq \frac{1}{n} (\Delta f)^2 + \rho \|\grad f \|^2_{\tensorg}.$$
In the above formula, $n = \dim M$, $\rho$ is a lower bound for the
Ricci curvature of~$M$ and~$f$ is any smooth function. In the notation
of Bakry and \'Emery~\cite{BaEm85}, this inequality is written as
$$\sfGamma_2(f) \geq \frac{1}{n} (L f)^2 + \rho \sfGamma(f), \qquad L = \Delta,$$
where\begin{align} \label{Sec1Gamma} \sfGamma(f,g) & = \frac{1}{2}
  \left(L (fg) - f L g - g L f \right) , \qquad & \sfGamma(f) =
  \sfGamma(f,f), \\ \label{Sec1Gamma2} \sfGamma_2(f, g) & = \frac{1}{2}
  \left(L \sfGamma(f,g) - \sfGamma(f ,L g) - \sfGamma(L f, g)\right) ,
  \qquad & \sfGamma_2(f) = \sfGamma_2(f,f).
\end{align}
For a good overview of results that follow from this inequality, see
\cite{Wan04} and references therein.

This approach has been generalized by F.~Baudoin and N.~Garofalo in
\cite{BaGa12} to sub-Riemannian manifolds with transverse
symmetries. A sub-Riemannian manifold is a connected manifold $M$ with
a positive definite metric tensor~$\tensorh$ defined only on a
subbundle~$\calH$ of the tangent bundle $TM$. As is typical, we will
assume that sections of~$\calH$ and their iterated Lie brackets span
the entire tangent bundle. This is a sufficient condition for the
sub-Riemannian structure $(\calH, \tensorh)$ to give us a metric
$\metricd_{cc}$ on $M$, where the distance between two points with
respect to $\metricd_{cc}$ is defined by taking the infimum of the
lengths of all curves tangent to~$\calH$ that connect the mentioned
points. For the definition of sub-Riemannian manifolds with transverse
symmetries, see \cite[Section 2.3]{BaGa12} or Part~I, Section 4.3. We
extended this formalism in Part~I to sub-Riemannian manifold with an
integrable metric-preserving complement,
consisting of all sub-Riemannian manifolds that can be obtained from
Riemannian foliations.

Given such a metric-preserving complement $\calV$ to $\calH$, there
exist a canonical corresponding choice of second order operator
$\roughL$ which locally satisfies
$$\roughL = \sum_{i=1}^n A_i^2 + \text{lower order terms}.$$
with $A_1, \dots, A_n$ being a local orthonormal basis of $\calH$. We
proved in Part~I that under mild conditions, there exist constants $n,
\rho_1, \rho_{2,0}$ and $\rho_{2,1}$ such that the operator satisfies
a generalized version of the curvature-dimension inequality
$$\sfGamma_2(f) + \ell \sfGamma^{\tensorv^*}_2(f) \geq \frac{1}{n} (Lf) 
+ (\rho_1 - \ell^{-1}) \sfGamma(f) + (\rho_{2,0} + \ell \rho_{2,1}) \sfGamma^{\tensorv^*}(f),$$
for any $f \in C^\infty(M)$ and $\ell >0$. Here, $\sfGamma(f)$ and
$\sfGamma_2(f)$ is defined as in \eqref{Sec1Gamma} and
\eqref{Sec1Gamma2} with $L = \roughL$, while $\sfGamma^{\tensorv^*}(f)
= \tensorv^*(df,df)$ for some $\tensorv^* \in \Gamma(\Sym^2 TM)$ and
$\sfGamma^{\tensorv^*}_2(f)$ is defined analogously to $\sfGamma_2(f)$.
We also gave a geometrical interpretation of these constants. A short
summary of the results of Part~I is given in
Section~\ref{sec:Summary}.

In this paper, we want to explore how this inequality can be used to
obtain results for the heat semigroup of $\roughL$. In
Section~\ref{sec:GradientBound}, we will address the question of
whether a smooth bounded function with bounded gradient under the action of the heat
semigroup will continue to have a uniformly bounded gradient. This
will be an important condition for the results to follow. For a
complete Riemannian manifold, a sufficient condition for this to hold
is that the Ricci curvature is bounded from below, see
e.g. \cite{Yau78} and \cite[Eq 1.4]{Tha98}. We are not able to give
such a simple formulation for the sub-Laplacian, however, we are able
to prove that it holds in many cases, including fiber bundles with
compact fibers and totally geodesic fibers. This was only previously
only known to hold for sub-Riemannian manifolds with transverse
symmetries of Yang-Mills type \cite[Theorem~4.3]{BaGa12}, along with some isolated
examples in \cite[Section 4]{Wan13} and \cite[Appendix]{BaWa14}. We give several
results using the curvature-dimension inequality of Part~I that only
rely on the boundedness of the gradient under the heat flow. Our results generalize theorems found in \cite{BaGa12,BaBo12,BBG14}. In
particular, if $\roughL$ is a sub-Laplacian on $(M, \calH, \tensorh)$
satisfying our generalized curvature-dimension inequality, then under
certain conditions (analogous to positive Ricci curvature in Riemannian
geometry) we have the following version of the Poincar\'e inequality
$$\| f - f_M \|_{L^2(M,\vol)} \leq \frac{1}{\sqrt{\alpha}}\, \| df\|_{L^2(\tensorh^*)}.$$
Here, $\alpha$ is a positive constant, $\tensorh^*$ is the co-metric of $(\calH, \tensorh)$, $f_M$ is the mean value of
a compactly supported function $f$ and for any $\eta \in \Gamma(T^*M)$ we use
$$\|\eta\|_{L^2(\tensorh^*)} := \int_M \tensorh^*(\eta, \eta) \dvol.$$

In Section~\ref{sec:GeodesicFibers} we look at results which require
the additional assumption that $\sfGamma^{\tensorv^*}(f,
\sfGamma(f)) = \sfGamma(f,
\sfGamma^{\tensorv^*}(f))$. This is important for inequalities
involving logarithms. We give a description of what this condition
means geometrically and discuss results that follow from it, such as a
Li-Yau type inequality and parabolic Harnack inequality.

In Section~\ref{sec:Comments} we give some concrete examples, mostly
focused on case of sub-Riemannian structures appearing from totally
geodesic foliations with a complete metric. Here, all previously
mentioned assumptions are satisfied. In this case, we also give a
comment on how the invariants in our sub-Riemannian
curvature-dimension inequality compare to the Riemannian curvature of
an extended metric.

In parallel with the development of our paper, part of the results of Theorem~\ref{th:CondARiemann} and Lemma~\ref{lemma:iff} was given in \cite{BKW14} for the case of sub-Riemannian obtained from Riemannian foliations with totally geodesic leaves that are of Yang-Mills type.

\subsection{Notations and conventions}
Unless otherwise stated, all manifolds are connected. If $\calE \to M$
is any vector bundle over a manifold $M$, its space of smooth sections
is written $\Gamma(\calE)$. If $s \in \Gamma(\calE)$, we generally
prefer to write $s|_x$ rather than $s(x)$ for its value in $x \in
M$. By a metric tensor~$\tensors$ on~$\calE$, we mean smooth section
of $\Sym^2 \calE^*$ which is positive definite or at least positive
semi-definite. For every such metric tensor, we write
$\|e\|_{\tensors} = \sqrt{\tensors(e,e)}$ for any $e \in \calE$ even
if $\tensors$ is only positive semi-definite. All metric tensors are
denoted by bold, lower case Latin letters (e.g. $\tensorh,
\tensorg,\dots$). We will only use the term Riemannian metric for a
positive definite metric tensor on the tangent bundle. If $\tensorg$
is a Riemannian metric, we will use $\tensorg^*$, $\wedge^k
\tensorg^*, \dots$ for the metric tensors induced on $T^*M,
\bigwedge^k T^*M, \dots$.

If $\alpha$ is a form on a manifold $M$, its contraction or interior
product by a vector field~$A$ will be denoted by either
$\iota_A\alpha$ or $\alpha(A, \newbullet).$ We use $\calL_A$ for the
Lie derivative with respect to $A$. If $M$ is furnished with a Riemannian metric $\tensorg$, any bilinear tensor $\tensors: TM \otimes TM \to \mathbb{R}$ can be identified with an endomorphism of $TM$ using $\tensorg$. We use the notation $\tr \, \tensors(\times, \times)$ for the trace of this corresponding endomorphism, with the metric being implicit. If $\calH$ is a subbundle of $TM$, we will also use the notation $\tr_{\calH} \tensors(\times, \times ) := \tr \tensors(\pr_{\calH} \times, \pr_{\calH} \times)$, where $\pr_{\calH}$ is the orthogonal projection to $\calH$.

\section{Summary of Part~I} \label{sec:Summary} In this section, we
briefly recall the most important definitions and results from Part~I.

\subsection{Sub-Riemannian manifolds}
A sub-Riemannian manifold is a triple $(M, \calH, \tensorh)$ where
$M$ is a connected manifold and $\tensorh$ is a positive definite
metric tensor defined only on the subbundle $\calH$ of
$TM$. Equivalently, it can be considered as a manifold with a positive
semi-definite co-metric $\tensorh^*$ that is degenerate along a
subbundle of $T^*M$. This latter mentioned subbundle will be
$\Ann(\calH)$, the annihilator of $\calH$, that consist of all covectors 
vanishing on $\calH$. Define $\shh\colon p \mapsto
\tensorh^*(p,\newbullet) \in \calH \subseteq TM$. We will assume that
the subbundle $\calH$ is bracket-generating, i.e. its sections and
their iterated brackets span the entire tangent bundle. Then we have a
well defined metric $\metricd_{cc}$ on $M$ by taking the infimum over the
length of curves that are tangent to $\calH$.

\subsection{Two notions of sub-Laplacian}
Let $\vol$ be any smooth volume form on $M$. We then define the
sub-Laplacian relative to the volume form $\vol$ as $\srL f = \dv \shh
df,$ where the divergence is defined relative to $\vol.$ From the
definition, it is clear that $\srL$ is symmetric relative the measure
$\vol$, i.e. $\int_M f \srL g \dvol = \int_M g \srL f \dvol$ for any
$f,g \in C^\infty_c(M)$ of compact support.

We also introduced the concept of a sub-Laplacian defined relative to
a complement $\calV$ of $\calH$. Let $\tensorg$ be any Riemannian
metric satisfying $\tensorg|_{\calH} = \tensorh$ and let $\calV$ be
the orthogonal complement of $\calH$. Consider the following
connection,
\begin{align} \label{rnabla} \rnabla_{A} Z = & \pr_{\calH}
  \nabla_{\pr_{\calH} A} \pr_{\calH} Z + \pr_{\calV}
  \nabla_{\pr_{\calV} A} \pr_{\calV} Z \\ \nonumber & + \pr_{\calH}
  [\pr_{\calV} A, \pr_{\calH} Z] + \pr_{\calV} [\pr_{\calH} A,
  \pr_{\calV} Z], \end{align} where $\nabla$ is the Levi-Civita
connection of $\tensorg$. We define the sub-Laplacian of $\calV$
as
$$\roughL = \tr_{\calH} \rnabla^2_{ \times, \times}f.$$
It is simple to verify that this definition is independent of
$\tensorg|_{\calV}$, it only depends on $\tensorh$ and the splitting
$TM = \calH \oplus \calV.$

\begin{remark} \label{re:LiftedL}
  If $\calV$ is the vertical bundle of a submersion $\pi: M \to B$
  into a Riemannian manifold $(B, \widecheck{\tensorg})$ and if
  $\tensorh$ is a sub-Riemannian metric defined by pulling back
  $\widecheck{\tensorg}$ to an Ehresmann connection $\calH$ on $\pi$,
  then the sub-Laplacian $\roughL$ of $\calV$ satisfies $$\roughL (f
  \circ \pi) = (\widecheck{\Delta} f) \circ \pi,$$ where
  $\widecheck{\Delta}$ is the Laplacian of $\widecheck{\tensorg}$ and
  $f \in C^\infty(B).$
\end{remark}

\subsection{Metric-preserving complement}
A subbundle $\calV$ is integrable if $[\Gamma(\calV), \Gamma(\calV)]
\subseteq \Gamma(\calV)$. By the Frobenius Theorem, such a subbundle
gives us a foliation on $M$. We say that an integrable complement
$\calV$ of $\calH$ is metric-preserving if
$$\calL_{V} \pr_{\calH}^* \tensorh =0, \quad \text{ for any } V \in \Gamma(\calV),$$
where $\pr_{\calH}$ is the projection
corresponding to the choice of complement $\calV$. Let $\tensorg$ be any Riemannian
metric such that $\tensorg|_{\calH} = \tensorh$ and $\calH^\perp = \calV$.
If we define $\rnabla$ as in \eqref{rnabla},
then $\calV$ is metric preserving if and only if $\rnabla \tensorh^* =
0$. The foliation of $\calV$ is then called a Riemannian foliation.

\subsection{Generalized curvature-dimension inequality}
For a given smooth second order differential operator $L$ without
constant term and for any section $\tensors^*$ of $\Sym^2 TM,$ define
\begin{align*}
  \sfGamma^{\tensors^*}(f,g) &= \tensors^*(df,dg), \quad & \sfGamma^{\tensors^*}(f,f) = \sfGamma^{\tensors^*}(f), \\
  \sfGamma^{\tensors^*}_2(f,g) &= \frac{1}{2} \left(L
    \sfGamma^{\tensors^*}(f,g) - \sfGamma(Lf, g) -
    \sfGamma^{\tensors^*}(f,Lg) \right), \quad &
  \sfGamma^{\tensors^*}_2(f,f) = \sfGamma^{\tensors^*}_2(f).
\end{align*}
Assume that
$$\frac{1}{2}  \left(L(fg) - f Lg - g Lf\right) = \tensorh^*(df, dg).$$
for some positive semi-definite section $\tensorh^*$ of $\Sym^2
TM$. We say that $L$ satisfies the generalized curvature-dimension
inequality \eqref{CDstar} if there is another positive semi-definite
section $\tensorv^*$ of $\Sym^2 TM$, a positive number $0 < n \leq
\infty$ and real numbers $\rho_1, \rho_{2,0}$ and $\rho_{2,1}$ such
that for any $\ell >0$ and $f \in C^\infty(M)$,
\begin{equation}
  \tag{\sf CD*} \label{CDstar}
  \sfGamma^{\tensorh^* + \ell \tensorv^*}_2(f) \geq \frac{1}{n} (Lf)^2 + (\rho_1 - \ell^{-1}) \sfGamma^{\tensorh^*}(f) + (\rho_{2,0} + \ell \rho_{2,1}) \sfGamma^{\tensorv^*}(f).
\end{equation}

Let $(M, \calH, \tensorh)$ be a sub-Riemannian manifold with $\calH$
being a bracket-generating subbundle of $TM$. Assume that we have an
integrable metric-preserving complement $\calV$ and let $\tensorg$ be
a Riemannian metric such that $\calH$ and $\calV$ are orthogonal, with
$\tensorh = \tensorg|_{\calH}$ and $\tensorv :=
\tensorg|_{\calV}$. Let $\tensorh^*$ and $\tensorv^*$ be their
respective co-metrics. Relative to these structures, we make the
following assumptions.
\begin{enumerate}[(i)]
\item We define the curvature of $\calH$ relative to the complement
  $\calV$ as the vector valued 2-form
$$\calR(A,Z) = \pr_{\calV} [\pr_{\calH} A, \pr_{\calH} Z], \qquad A, Z \in \Gamma(TM).$$
We assume that there is a finite, minimal positive constant $\McalR < \infty$ such that $\|\calR(v, \newbullet)\|_{\tensorg^* \otimes \tensorg} \leq \McalR\|\pr_{\calH} v\|_{\tensorg}$ for any $v \in TM$. Since $\McalR$ is never
zero when $\calV \neq 0$, we can normalize $\tensorv$ by requiring
$\McalR =1.$ Let $\mcalR$ be the maximal constant satisfying $\|
\alpha(\calR(\newbullet, \newbullet))\|_{\wedge^2 \tensorh^*} \geq
\mcalR \|\alpha \|_{\tensorv^*}$ pointwise for any $\alpha \in \Gamma(T^*M)$.
Note that $\mcalR$ can only be non-zero if $\calH$ is bracket-generating
of step~2, i.e. if $\calH$ and its first order brackets span the entire
tangent bundle.
\item Define $\Ric_{\calH}(Z_1,Z_2) = \tr \big(A \mapsto
    R^{\rnabla}(\pr_{\calH} A,Z_1)Z_2 \big)$. This is a symmetric
  2-tensor, which vanishes for vectors in $\calV$. We assume that
  there is a lower bound $\rRicH$ for $\RicH$, i.e. for every $v \in
  TM$, we have $$\RicH(v,v) \geq \rRicH \|\pr_{\calH}
  v\|_{\tensorh}^2.$$
\item Write $\MII = \sup_{M} \big\| \rnabla_{\newbulletsub}^{\mathstrut}
  \tensorv^* (\newbullet,\newbullet) \big\|_{\tensorg^* \otimes \Sym^2
    \tensorg^*}$ and assume that it is finite. Define
$$(\roughL \tensorv^*)(\alpha,\alpha) = \tr_{\calH} (\rnabla_{\times,\times }^2 \tensorv^*)(\alpha,\alpha)$$
and assume that $(\roughL \tensorv^*)(\alpha,\alpha) \geq \rLv
\|\alpha\|_{\tensorv^*}^2$ pointwise for any $\alpha \in \Gamma(T^*M)$.
\item Finally, introduce $\RicHV$ as
$$\RicHV (A,Z) = \frac{1}{2} \tr\left( \tensorg(A, (\rnabla_{\times}^{\mathstrut} \calR)(\times,Z)) 
+ \tensorg(Z, (\rnabla^{\mathstrut}_{\times} \calR)(\times,A)) \right).$$
Assume then that $\RicHV(Z,Z) \geq - 2\MRicHV \|\pr_{\calV} Z\|_{\tensorv} \|\pr_{\calH} Z\|_{\tensorh}$ pointwise.
\end{enumerate}

These assumptions guarantee that the sub-Laplacian $\roughL$ of
$\calV$ satisfies \eqref{CDstar}.
\begin{theorem} \label{th:CD} Define $\sfGamma^{\tensors^*}_2$ with
  respect to $L = \roughL$. Then $\roughL$ satisfies \eqref{CDstar}
  with
  \begin{equation} \label{rhoSR}
    \left\{ \begin{aligned} 
    n & = \rank \calH, \\
    \rho_1 & = \rRicH - c^{-1}, \\
    \rho_{2,0} & =   \frac{1}{2} \mcalRvar^2 - c(\MRicHV + \MII)^2, \\
        \rho_{2,1} & = \frac{1}{2} \rLv - \MII^2, \end{aligned} \right.
  \end{equation}
  for any positive $c > 0$.
\end{theorem}
See Part~I, Section~3.2 for a geometric interpretation of these constants.

\subsection{The case when $\rnabla$ preserves the metric} Assume that we can
find a metric tensor $\tensorv$ on $\calV$ satisfying $\rnabla
\tensorv^* = 0$. Then $\roughL = \srL$, where $\srL$ is defined
relative to the volume form of the Riemannian metric $\tensorg$
defined by $\tensorg^* = \tensorh^* + \tensorv^*.$ Hence, $L = \srL $
is symmetric with respect to this volume form and satisfies the
inequality
\begin{equation} \tag{\sf CD} \label{CD} \sfGamma^{\tensorh^* + \ell
    \tensorv^*}_2(f) \geq \frac{1}{n} (L f)^2 + \left(\rho_1 -
    \ell^{-1} \right) \sfGamma^{\tensorh^*}(f) + \rho_{2}
  \sfGamma^{\tensorv^*}(f) 
\end{equation}
with
\begin{equation} \label{rhoSR2} \left\{ \begin{aligned} \rho_1 & =
      {\rRicH} - c^{-1}, \\ \rho_{2} & = \frac{1}{2}
      \mcalRvar^2 - c\MRicHV^2, \end{aligned} \right.
\end{equation}
for any positive $c > 0$. We shall also need the following result.
\begin{proposition} \label{prop:DoubleGamma} For any $f \in
  C^\infty(M)$, and any $c >0$ and $\ell >0$,
  \begin{align*}
    \frac{1}{4} \sfGamma^{\tensorh^*}(\sfGamma^{\tensorh^*}(f)) &\leq  \sfGamma^{\tensorh^*}(f) \left(\sfGamma^{\tensorh^* + \ell \tensorv^*}_2(f) - (\varrho_1 -\ell^{-1}) \sfGamma^{\tensorh^*}(f) - \varrho_2 \sfGamma^{\tensorv^*}(f) \right), \\
    \frac{1}{4} \sfGamma^{\tensorh^*}(\sfGamma^{\tensorv^*}(f)) &\leq
    \sfGamma^{\tensorv^*}(f) \sfGamma^{\tensorv^*}_2(f),
  \end{align*}
  where $\varrho_1 = \rRicH - c^{-1}$ and $\varrho_2 = - c \MRicHV^2.$
\end{proposition}

\subsection{Spectral Gap}
Let $(M, \calH, \tensorh)$ be a compact sub-Riemannian manifold where
$\calH$ is bracket-gen\-er\-at\-ing. Let $L$ be a smooth second order
operator without constant term satisfying $\tensorq_L = \tensorh^*$
and assume also that $L$ is symmetric with respect to some volume form
$\vol$ on $M$.  Assume that $L$ satisfies \eqref{CDstar} with
$\rho_{2,0} > 0$. Let $\lambda$ be any nonzero eigenvalue of $L$. Then
$$\frac{n \rho_{2,0}}{n + \rho_{2,0}(n-1)} \left(\rho_1 - \frac{ k_2}{\rho_{2,0}} \right) \leq - \lambda, \qquad k_2 = \max \{0, -\rho_{2,1} \}.$$

\section{Results under conditions of a uniformly bounded
  gradient} \label{sec:GradientBound}
\subsection{Diffusions of second order
  operators} \label{sec:GeneralOperator} Let $T^2M$ denote the bundle
of second order tangent vectors. Let $L$ be a section of $T^2M$,
i.e. a smooth second order differential operator on $L$ without
constant term. Consider the short exact sequence
$$0 \to TM \stackrel{\inc}{\longrightarrow} T^2M \stackrel{\tensorq}{\longrightarrow} \Sym^2 T M \to 0$$
where $\tensorq_L = \tensorq(L)$ is defined by
\begin{equation} \label{qform} \tensorq_L(df,dg) = \frac{1}{2}
  \left(L(fg) - f Lg - g Lf \right), \quad f,g \in
  C^\infty(M).\end{equation} Assume that $\tensorq_L$ is
positive semi-definite. Then for any point $x \in M$ and relative to some
filtered probability space $(\Omega, \scrF_{\!\newbulletsub}^{\mathstrut}, \prob)$, we have a
$\frac{1}{2} L$-diffusion $X = X(x)$ defined up to some
explosion time $\tau = \tau(x)$, see \cite[Theorems 1.3.4 and
1.3.6]{Hsu02}. In other words, there exist an $\scrF_{\!\newbulletsub}^{\mathstrut}$-adapted
$M$-valued semimartingale $X(x)$ satisfying $X_0(x) = x$ and
such that for any $f \in C^\infty(M)$, 
$$d(f(X_t) ) -\frac{1}{2} L f(X_t) \, dt$$ is 
the differential of a local martingale up to $\tau(x)$. The diffusion $X(x)$ is defined on the
stochastic interval 
$[0,\tau(x))$, with $\tau(x)$ being an explosion time in the sense that the event
$\{\tau(x) < \infty\}$ is almost surely contained in $\{\lim_{t \uparrow \tau} X_t(x)
= \infty\}$. 
For a construction of $X_t(x)$ in the case of $L = \roughL$, see Part~I, Section~2.5.

Let $P_t $ be the corresponding semigroup $P_t f(x) =
\expect[ 1_{t \leq \tau} f( X_t(x))]$ for bounded measurable
functions $f$. Note that in general $P_t 1\leq1$ with equality if and
only if $\tau(x) = \infty$ a.s. Also note that for any compactly
supported $f \in C_c^\infty(M)$, we have $\partial_t P_t f =
\frac{1}{2} L P_t f$. If $\tau = \infty$ a.s., then $u_t = P_t f$
is the unique solution to $\partial_t u_t =\frac{1}{2} L u_t$ with
initial condition $u_0 = f$, where $(t,x) \mapsto u_t(x)$ is a smooth
function on ${\real_+} \times M$.

Since $\tensorq_L$ is positive semi-definite, we can write $L$
(non-uniquely) as
$$L = \sum_{i=1}^k Z_i^2 + Z_0,$$
where $k$ is an integer and $Z_0, Z_1, \dots, Z_k$ are vector fields,
not necessarily linearly independent at every point. If we assume that
these vector fields and their brackets span the entire tangent bundle,
then $L$ is a hypoelliptic operator \cite{Hor67}. Hence, it has a
smooth heat kernel with respect to any volume form on $M$. By \cite{StVa72}, we also
have $P_t f > 0$ for any nonnegative function $f \in C^\infty(M)$, not identically zero,
see also \cite[Introduction]{Kun78}. We will \textit{only} consider such second order operators
in this paper.

Write $\tensorh^* = \tensorq_L$. Assume that $L$ satisfies \eqref{CDstar} for some $\tensorv^*$. We
want to use this inequality to obtain statements of $P_t$. However, we
are going to need the following condition to hold to make such
statements.

\subsection{Boundedness of the gradient under the action of the heat semigroup}
The most important property which we are going to need for all of our
results, is the following condition. 
Let $C_b^\infty(M)$ be the collection of all bounded smooth functions.
\begin{equation}
  \tag{\sf A} \label{CondA} \begin{array}{c} 
   \text{We have $P_t 1 = 1$ and for any $f \in C^\infty_b(M)$ with $\sfGamma^{\tensorh^* + \tensorv^*}(f) \in C_b^\infty(M)$} \\
   \text{and any $T > 0$, it holds that} 
    \sup_{t \in [0,T]} \|\sfGamma^{\tensorh^* + \tensorv^*}(P_tf)\|_{L^\infty} < \infty. \end{array}
\end{equation}
To understand condition \eqref{CondA} better, let us first discuss the
special case when $\tensorh = \tensorg$ is a Riemannian metric,
$\tensorv^* = 0$ and $L = \Delta$ is the Laplacian of $\tensorg$. Then
\eqref{CDstar} holds if and only if the Ricci curvature is bounded
from below, see e.g. \cite{Wan04}. If we in addition know that
$\tensorg$ is complete, then \eqref{CondA} is satisfied. However, even if we know that
$P_t 1 =1$ and that the manifold is flat, condition \eqref{CondA} still may not hold if
$\tensorg$ is an incomplete metric. See~\cite{Tha98}
for a counter-example.

We list some cases where we are ensured that \eqref{CondA} is
satisfied. We expect there to be more cases where this condition holds.

\subsubsection{Fiber bundles with compact fibers} \label{sec:FiberCompact} 
Let $L$ be a second order operator
on a manifold $M$ with $\tensorq_L = \tensorh^*$. Let $\tensorv^*$ be
any other co-metric such that $\tensorh^* + \tensorv^*$ is positive
definite. The following observation was given in \cite[Lemma~2.1,
Proof (i)]{Wan13}.
\begin{lemma} \label{WangF} Assume that there exists a function $F \in
  C^\infty(M)$ and a constant $C >0$ satisfying
  \begin{enumerate}[$\ \bullet$]
  \item $\{x :F(x) \leq s\} $ is compact for any $s > 0$,
  \item $LF \leq C F$,
  \item $\sfGamma^{\tensorh^* + \tensorv^*}(F) \leq C F^2.$
  \end{enumerate}
  Then \eqref{CondA} holds for the semigroup $P_t$ of the diffusion of $ L$.
\end{lemma}
Let $(B, \widecheck{\tensorg})$ be a complete $n$-dimensional
Riemannian manifold with distance $\metricd_{\widecheck{\tensorg}}$
and Ricci bound from below by $\rho \leq 0$. For a given point $b_0
\in B$, define $r = \metricd_{\widecheck{\tensorg}}(b_0,
\newbullet)$. Then the function $F= \sqrt{1+ r^2}$ (or rather an
appropriately smooth approximation) satisfies the above conditions
relative to $\widecheck{\Delta}$. This follows from the fact that
(outside the cut-locus) $\sfGamma^{\widecheck{\tensorg}}(r) = 1$ and from the Laplacian
comparison theorem
$$\widecheck{\Delta} r \leq (n-1) \left(\frac{1}{r} + \sqrt{- \rho} \right).$$
Now let $\pi: M \to B$ be a fiber bundle with a compact fiber over
this Riemannian manifold $B$. Choosing an Ehresmann connection $\calH$
on $\pi$, we define a sub-Riemannian manifold $(M, \calH, \tensorh)$
by $\tensorh = \pi^* \widecheck{\tensorg}|_{\calH}$. Then $F \circ \pi$ clearly
satisfies Lemma~\ref{WangF} with respect to $L = \roughL + Z$ where
$\roughL$ is the sub-Laplacian of $\calV = \ker \pi_*$ and $Z$ is any
vector field with values in $\calV$. It follows that \eqref{CondA}
holds in this case.

\begin{remark} Let $\pi\colon M \to B$ be a surjective submersion into a Riemannian manifolds $(B, \widecheck{\tensorg})$. Let $\calH$ be an Ehresmann connection on $\pi$ and define a sub-Riemannian
structure~$(\calH, \tensorh)$ by $\tensorh = \pi^* \widecheck{\tensorg}|_{\calH}$.
In this case, $\pi$ is a distance-decreasing map from the metric space $(M, \metricd_{cc})$
to $(M, \metricd_{\widecheck{\tensorg}})$, where the metrics $\metricd_{cc}$ and
$\metricd_{\widecheck{\tensorg}}$ are defined relative to $(\calH, \tensorh)$ and
$\widecheck{\tensorg}$, respectively. This follows from the observation that for any horizontal curve $\gamma$
in~$M$ from the point $x$ to the point $y$, the curve $\pi \circ \gamma$ will be
a curve of equal length in $B$ connecting $\pi(x)$ with $\pi(y)$, hence $\metricd_{cc}(x, y) \geq \metricd_{\widecheck{\tensorg}}(\pi(x), \pi(y))$. In particular, if $\metricd_{cc}$ is complete, so is $\metricd_{\widecheck{\tensorg}}$, and the converse also hold if $\pi$ is a fiber bundle with compact fibers.

Furthermore, if $\srL$ is the sub-Laplacian of $\calV = \ker \pi_*$ satisfying \eqref{CDstar}, then the Ricci curvature
of $B$ is bounded from below, since, by Remark~\ref{re:LiftedL}, if we insert a function $f\circ \pi$, $f \in C^\infty(B)$ into \eqref{CDstar}, we obtain the usual curvature-dimension inequality on $B$,
$$\sfGamma^{\widecheck{\tensorg}}_2(f) \geq \frac{1}{n} (\widecheck{\Delta})^2 + \rho_1 \sfGamma^{\widecheck{\tensorg}}(f).$$
A result in \cite[Prop~6.2]{Bak94} tells us that $\rho_1$ must be a lower Ricci bound for $B$.
\end{remark}

We summarize all the above comments in the following proposition.
\begin{proposition}
Let $(M, \calH, \tensorh)$ be a complete sub-Riemannian manifold with an integrable
metric preserving complement $\calV$. Let $\calF$ be the foliation induced by $\calV$ and
let $\roughL$ be the sub-Laplacian of $\calV$. Assume that the leafs of $\calF$ are compact
and that $M/\calF$ gives us a well defined smooth manifold. 
Finally assume that $L =\roughL + Z$ satisfies \eqref{CDstar}
with respect to some $\tensorv^*$ on $\calV$.
Then \eqref{CondA} also hold for the corresponding semigroup $P_t$ of $L$.
\end{proposition}
Notice that in this case, unlike what we will discuss next, there is no requirement on the number of brackets needed of vector fields in $\calH$ in order to span the entire tangent bundle.

\subsubsection{A sub-Laplacian on a totally geodesic Riemannian
  foliation} \label{sec:sLCompleteR} Assume that $(M, \tensorg)$ is a
complete Riemannian manifold with a foliation $\calF$ given by an
integrable subbundle $\calV$. Let $\calH$ be the orthogonal complement
of $\calV$ and assume that $\calH$ is bracket-generating. Write $\tensorh = \tensorg|_{\calH}$. Define
$\rnabla$ relative to the splitting $TM= \calH \oplus \calV$ as
in~\eqref{rnabla}. Assume that $\rnabla \tensorg = 0$, which is equivalent to stating
that $\calV$ is a metric preserving complement of $(M, \calH, \tensorh)$
and that $\calF$ is a totally geodesic foliation. Note that since~$\tensorg$
is complete, so is $(M, \metricd_{cc})$, where
$\metricd_{cc}$ is defined relative to the sub-Riemannian metric~$\tensorh$.
For such sub-Riemannian manifolds, the we can deduce the following.

\begin{theorem} \label{th:CondARiemann}
Let $\srL$ be the sub-Laplacian of the volume form of~$\tensorg$ or equivalently~$\calV.$
Assume that $\srL$ satisfies
  the assumptions of Theorem~\textup{\ref{th:CD}} with $\mcalR
  >0$. Let $k = \max\{ - \rRicH, \MRicHV^2\} \geq 0$. Then, for any compactly supported $f
  \in C_c^\infty(M)$, $\ell > 0$ and $t \geq 0$,
  \begin{align*}\sqrt{\sfGamma^{\tensorv^*}(P_tf)} &\leq P_t \sqrt{\sfGamma^{\tensorv^*}(f)}\,,\\
    \sqrt{\sfGamma^{\tensorh^*}(P_tf)} + \sfGamma^{\tensorv^*}(P_t f)
    &\leq e^{k t/2} P_t\left( \sqrt{\sfGamma^{\tensorh^*}(f)} +
      \sfGamma^{\tensorv^*}(f) \right) + \frac{2}{k} ( e^{k t/2} - 1),
  \end{align*}
  where we interpret $\frac{2}{k} ( e^{k t/2} - 1)$ as $t$ when $k
  =0$. As a consequence \eqref{CondA} holds.
  
  In particular, any $\frac{1}{2} \srL$-diffusion $ X(x)$ with $X_0(x) =x \in M$ has infinite lifetime.
\end{theorem}

We remind the reader that $\mcalR >0$ can only happen if $TM$ is spanned by
$\calH$ and first order brackets of its sections.
The proof is similar to the proof given for the special case of sub-Riemannian manifolds
with transverse symmetries of Yang-Mills type given in \cite[Section~3
\& Theorem~4.3]{BaGa12}. In our terminology, these are sub-Riemannian manifolds
with a trivial, integrable, metric-preserving complement $\calV$ satisfying $\MRicHV = 0$.
The key factors that allow us to use a similar
approach are Proposition~\ref{prop:DoubleGamma} and the relation
$[\srL, \Delta] f = 0$, where $\Delta$ is the Laplace operator of $\tensorg$ and $f
\in C^\infty(M)$. The latter results follow from
Lemma~\ref{lemma:srLDeltaCommute}~(c) in the Appendix. Since the proof
uses spectral theory and calculus on graded forms, it is left to
Appendix~\ref{sec:ProofCondARiemann}. Theorem~\ref{th:CondARiemann} also holds in some cases
when $\calV$ is not an integrable subbundle. See Appendix~\ref{sec:NotIntegrable} for details.

\subsection{General formulation}
Let $L$ be an operator as in Section~\ref{sec:GeneralOperator} with
corresponding $\frac{1}{2} L$-diffusion $X(x)$ satisfying $X_0(x) = x$ and semigroup $P_t$.
We will assume that $L$ satisfies
\eqref{CDstar} with $\tensorv^*$ and the constants $n, \rho_1,
\rho_{2,0}$ and $\rho_{2,1}$ being implicit. Note that if $L$ satisfies
\eqref{CDstar} for some value of the previously mentioned constants,
then $L$ also satisfies the same inequality for any larger $n$ or
smaller values of $\rho_1, \rho_{2,0}$ or $ \rho_{2,1}$. For the
remainder of the section, no result will depend on $n$, however, we
will need condition \eqref{CondA} to hold.

Our proofs rely on the fact that, for any smooth function $$(t,x)
\mapsto u_t(x) \in C^\infty([0,\infty) \times M, \real),$$ we have 
a stochastic process $Y_t = u_t \circ X_t$ such that
$dY_t$ equals $\left((\partial_t + L) u_t\right) \circ X_t\, dt$ modulo differentials of local martingales. Hence, if $(\partial_t + L)
u_t \geq 0$ and if $u_t(\newbullet)$ is bounded for
every fixed $t$, then $Y_t$ is a (true) submartingale and $\expect[
Y_t]$ is an increasing function with respect to $t$.

In our presentation, we will usually state the result for a smooth,
bounded function $f \in C_b^\infty(M) $ with bounded gradient 
$\sfGamma^{\tensorh^* + \tensorv^*}(f) \in C^\infty_b(M)$.
Our results generalize theorems found in \cite{BaGa12,BaBo12,BBG14}.

We will first construct a general type of inequality, from which many
results can be obtained. See \cite[Theorem 1.1 (1)]{Wan13} for a similar
result, with somewhat different assumptions.

\begin{lemma} \label{lemma:ALambdaC} Assume that $L$ satisfies the conditions
  \eqref{CDstar} and \eqref{CondA}. For any $T >0$, let $a, \ell \in
  C([0,T], \real)$ be two continuous functions which are smooth and
  positive on~$(0,T)$. Assume that there exist a constant $C$, such
  that
  \begin{equation} \label{ALambdaC} \dot a(t) + \left( \rho_1 -
      \frac{1}{\ell(t)} \right) a(t) + C \geq 0, \quad \dot \ell(t) +
    \rho_{2,0} + \left(\rho_{2,1} + \frac{\dot a(t)}{a(t)} \right)
    \ell(t) \geq 0,
  \end{equation}
  holds for every $t \in (0,T)$. Then
  \begin{align*}
    a(0) \sfGamma^{\tensorh^* + \ell(0) \tensorv^*}(P_T f) \leq a(T)
    P_T \sfGamma^{\tensorh^* + \ell(T) \tensorv^*}(f) + C \left( P_T
      f^2 - (P_T f)^2 \right)
  \end{align*}
  for any $f \in C^\infty_b(M)$ with $\sfGamma^{\tensorh^*+\tensorv^*}(f) \in C_b^\infty(M)$.
\end{lemma}

\begin{proof}
  Define $u_t$ by $u_t(x) = P_{T-t} f(x)$ for any $0 \leq t \leq T$, $x\in M$. 
For any $x \in M$, consider the stochastic process
$$Y_t(x) := a(t) \sfGamma^{\tensorh^* + \ell(t) \tensorv^*}(u_t) \circ X_t(x) + C u_t^2 \circ X_t(x).$$
Write $\equivM$ for equivalence modulo differentials of local
martingales. Then, if \eqref{ALambdaC} holds
\begin{align*}
  dY_t \equivM & \left(\dot a(t) \sfGamma^{\tensorh^* + \ell(t) \tensorv^*}(u_t) + a(t) \dot \ell(t) \sfGamma^{\tensorv^*}(u_t) + C \sfGamma^{\tensorh^*}(u_t) \right)\circ X_t dt \\
  & + a(t) \sfGamma^{\tensorh^* + \ell(t) \tensorv^*}_2(u_t) \circ X_t dt \\
  \geq & \left( \dot a(t) + ( \rho_1 - \ell(t)^{-1} ) a(t) + C\right) \sfGamma^{\tensorh^*}(u_t) \circ  X_t dt \\
  & + a(t) \left( \dot \ell(t) + \frac{\dot a(t)}{a(t)} + \rho_{2,0} +
    \rho_{2,1} \ell(t) \right) \sfGamma^{\tensorv^*}(u_t) \circ X_t dt \geq 0.
\end{align*}
Since $Y_t$ is bounded by \eqref{CondA}, it is a true
submartingale.  Hence
\begin{align*}
  \expect[ Y_T] & = a(T) P_T \sfGamma^{\tensorh^* + \ell(T) \tensorv^*}(f) + C P_T f^2 \\
  & \geq \expect[Y_0] = a(0) \sfGamma^{\tensorh^* + \ell(0)
    \tensorv^*}(P_T f) + C (P_T f)^2.
\end{align*}\qed
\end{proof}

\subsection{Gradient bounds}
We give here the first results that follow from
Lemma~\ref{lemma:ALambdaC}.

\begin{proposition} \label{prop:GradBound} Assume that $L$ satisfies conditions
  \eqref{CDstar} and \eqref{CondA}. Let $f \in C^\infty_b(M) $ be any smooth bounded function satisfying $\sfGamma^{\tensorh^* + \tensorv^*}(f) \in C_b^\infty(M)$.
  \begin{enumerate}[\rm(a)]
  \item For any constant $\ell > 0$, if $\alpha(\ell) =
    \min\left\{\rho_1 - \frac{1}{\ell}, \rho_{2,1} +
    \frac{\rho_{2,0}}{\ell} \right\},$ then
$$\sfGamma^{\tensorh^* + \ell \tensorv^*}(P_t f) \leq e^{-\alpha(\ell) t} P_t \sfGamma^{\tensorh^* + \ell \tensorv^*}(f).$$
\item Assume that $\rho_{2,0} >0$ and let $k_1 = \max\{ 0,
  -\rho_1\}$ and $k_2 = \max \{ 0, - \rho_{2,1}\}$.  Then
  \begin{align*} t \sfGamma^{ \tensorh^*}(P_t f) \leq \left(1 +
      \frac{2}{\rho_{2,0}} + \left(k_1 + \frac{k_2}{\rho_2} \right)
      t \right) (P_t f^2 - (P_t f)^2) .  \end{align*}
\item Assume that $\rho_1 \geq 0$, $\rho_{2,1} \geq 0$ and $\rho_{2,0} >0$. Then
$$\frac{1 - e^{-\rho_1 t}}{\rho_1} \sfGamma^{\tensorh^*}(P_tf) \leq \left(1 + \frac{2}{\rho_{2,0}}  \right) \left( P_t f^2 - (P_t f)^2 \right),$$
where we interpret $(1- e^{-\rho_1 t})/{\rho_1}$ as $t$ when $\rho_1=0$.
\item Assume that $\rho_1$, $\rho_{2,0}$ and $\rho_{2,1}$ are
  nonnegative. Then for any $\ell >0$
$$ \frac{\ell}{\ell +t } \left(P_t f^2 - (P_t f)^2 \right) \leq t P_t \sfGamma^{\tensorh^* + (\ell+t) \tensorv^*}(f).$$
\end{enumerate}
\end{proposition}

\begin{proof}
  For all of our results (a)--(c), we will use
  Lemma~\ref{lemma:ALambdaC}.
  \begin{enumerate}[(a)]
  \item Let $\ell(t) = \ell$ be a constant, choose $C =0$ and put
    $a(t) = e^{-\alpha(\ell) t}$. Then \eqref{ALambdaC} is satisfied
    and we obtain
$$\sfGamma^{\tensorh^* + \ell \tensorv^*}(P_Tf) \leq e^{-\alpha(\ell) T} P_T\sfGamma^{\tensorh^* + \ell \tensorv^*}(P_T f).$$
\item For any $T \geq0 $, consider $a(t) = T-t$ and $\ell(t) =
  \frac{\rho_{2,0}}{T k_2 +2} (T-t)$. Then
  \begin{align*} & \dot \ell(t) + \rho_{2,0} + \left(\rho_{2,1} +
      \frac{\dot a(t)}{a(t)} \right) \ell(t) \geq 0, \end{align*} and
$$\dot a(t) + (\rho_1  - \ell(t)^{-1})a(t) \geq -1 -k_1 - \frac{Tk_2 +2}{\rho_{2,0}},$$
so \eqref{ALambdaC} is satisfied if we define $C= 1 + k_1 T + \frac{T
  k_2 +2}{\rho_{2,0}}$. Using Lemma~\ref{lemma:ALambdaC}, we obtain
\begin{align*} T \sfGamma^{\tensorh^* + \frac{\rho_{2,0}}{T k_2 +2} T
    \tensorv^*}(P_T f) \leq C (P_T f^2) - C (P_T f)^2 .  \end{align*}
\item Since the case $\rho_1 =0$ is covered in (b), we can assume $\rho_1 > 0$. Define
$$a(t) = \frac{1 - e^{-\rho_1 (T-t)}}{\rho_1}$$
and let
$$\ell(t) = \rho_{2,0} \frac{\int_t^T a(s) \, ds }{a(t)} = \rho_{2,0} \frac{e^{-\rho_1(T-t)} -1 + \rho_1 (T-t)}{\rho_1 (1- e^{-\rho_1 (T-t)})} .$$
Note that $\lim_{t \uparrow T} \ell(t) = 0$, while $\lim_{t\uparrow T} {a(t)}/{\ell(t)} ={2}/{\rho_{2,0}}$.
The latter number is also an upper bound for ${a(t)}/{\ell(t)}$ since
$$\frac{d}{dt} \frac{a(t)}{\ell(t)} = \frac{a(t) \left(2  \dot a(t) \int_t^T a(s) \, ds + a(t)^2 \right)}{\rho_{2,0} (\int_t^T a(s) ds)^2} > 0$$
from the fact that
\begin{align*} \qquad& 2  \dot a(t) \int_t^T a(s) ds + a(t)^2 \\  
&= \frac{1}{\rho^2_1} \left( - 2 e^{-\rho_1(T-t)} \left(e^{-\rho_1(T-T)} - 1 + \rho_1(T-t)  \right) + (1-e^{-\rho_1(T-t)})^2 \right) \\
& = \frac{1}{\rho_1^2} \left( - 2\rho_1 (T-t) e^{-\rho_1(T-t)} + 1-e^{-2\rho_1(T-t)} \right) .\end{align*}
and that $s \mapsto 1 -e^{-2s} -2x e^{-2s}$ is an increasing function, vanishing at $s = 0$.
We can then define $C = 1 + \frac{2}{\rho_{2,0}} $ such that $a(t), \ell(t)$ and $C$ satisfies \eqref{ALambdaC}.

\item Define $a(t) = t$, $\ell(t) = \frac{(\ell+T) t}{T}$ and $C =
  - \frac{\ell}{\ell +T}$, then \eqref{ALambdaC} is satisfied.\qed
\end{enumerate}
\end{proof}
We see here that the results of (a) and (d) cannot be stated
independently of a choice of co-metric $\tensorv^*$. However, in the case of (a),
this does help us to get global statements that are independent of $\tensorv^*$.

\subsection{Bounds for the $L^2$-norm of the gradient and the Poincar\'e inequality}
We want to use an approach similar to what is used in \cite[Corollary
2.4]{BaBo12} to obtain a global inequality from the pointwise
estimate in Proposition~\ref{prop:GradBound}~(a) which is independent
of~$\tensorv^*$.
\begin{lemma} \label{lemma:RemoveV} Let $L \in \Gamma(T^2M)$ be a
  second order operator without constant term and with $\tensorq_L =
  \tensorh^*$ positive semi-definite. Assume also that there exists a volume form
  $\vol$, such that
$$\int_M f L g \dvol = \int_M g Lf \dvol, \quad f,g \in C_c^\infty(M),$$
and that $L$ is essentially self-adjoint on compactly supported functions $C_c^\infty(M)$.

Let $P_tf$ be the semigroup defined as in Section~\textup{\ref{sec:GeneralOperator}} 
and let $b:C^\infty_c(M) \times [0,\infty) \to \real$ be any function such that 
\begin{equation} \label{L1gradientFull} \| \sfGamma^{\tensorh^*}(P_t
  f) \|_{L^1} \leq b(f,t), \quad \text{for any } f \in C_c^\infty(M),\ 
  t >0.
\end{equation}
Assume that $\beta(f,t) := \lim_{T \to \infty} b(f,T)^{t/T}$
exist for every $t > 0$. Then
$$\| \sfGamma^{\tensorh^*} (P_t f) \|_{L^1} \leq \beta(f,t) \|\sfGamma^{\tensorh^*}(f)\|  \quad \text{for any } f \in C_c^\infty(M). $$
\end{lemma}

\begin{proof}
  Denote the unique self-adjoint extension of $L$ an operator on
  $L^2(M, \vol)$ by the same letter, and let $\Dom(L)$ be its
  domain. Then $e^{t/2 L} f$ is the unique solution in $L^2(M,\vol)$ of
  equation $\partial_t u_t = \frac{1}{2} L u_t$ with initial
  condition $u_0 = f \in C_c^\infty(M)$. Since $P_t f$ is in $L^2(M,\vol)$ whenever
  $f$ is in $L^2(M,\vol)$, we have $P_t f = e^{t/2L}f$ (see Appendix~\ref{sec:Spectral}
  for more details).

  Notice that since $\tensorq_L = \tensorh^*$ is positive
  semi-definite, the self-adjoint operator $L$ is nonpositive. Let
  $\langle \newbullet, \newbullet \rangle$ denote the inner product on
  $L^2(M,\vol)$. Consider the spectral decomposition $L = -
  \int_0^\infty \lambda dE_{\lambda}$. Then since
  $\|\sfGamma^{\tensorh^*}(f)\|_{L^1} = - \langle f, L f \rangle$,
  while $$\|\sfGamma^{\tensorh^*}(P_tf)\|_{L^1} = -\langle f , L P_{2t}
  f \rangle,$$ the H\"older inequality tells us that for any $0 < t < T$,
  \begin{align*}
    \| \sfGamma^{\tensorh^*}(P_tf) \|_{L^1} &= \int_0^\infty \lambda e^{-t\lambda} d\langle E_\lambda f, f\rangle \\
    &\leq \left(\int_0^\infty \lambda e^{-\lambda T} \, d\langle E_\lambda f, f\rangle\right)^{t/T}  \left( \int_0^\infty \lambda \, d\langle E_\lambda f, f \rangle \right)^{(T-t)/T} \\
    &\leq b(f,T)^{t/T} \|\sfGamma^{\tensorh^*}(f)\|_{L^1}^{T/(T-t)} .
  \end{align*}
  Let $T \to \infty$ for the result.\qed
\end{proof}

We combine this result with the curvature-dimension inequality.
\begin{proposition} \label{prop:Poincare} Let $L$ be any second order
  operator such that the Carnot-Cara\-th\'eo\-dory metric $\metricd_{cc}$
  defined by the sub-Riemannian co-metric $\tensorh^* := \tensorq_L$
  is complete. Assume that $L$ satisfies \eqref{CDstar} and that
  \eqref{CondA} holds. Assume also that $L$ is symmetric with respect
  to any volume form $\vol$, i.e.~$\int_M fLg \dvol = \int_M gLf \dvol$ for
  any $f,g \in C_c^\infty(M)$.
  \begin{enumerate}[\rm(a)]
  \item For any $f \in C_c^\infty(M)$,
$$\| \sfGamma^{\tensorh^*}( P_t f) \|_{L^1} \leq e^{-kt} \| \sfGamma^{\tensorh^*}(f)\|_{L^1},$$
where $k = \min\{\rho_1, \rho_{2,1} \}$.
\item Assume that $\rho_1 \geq \rho_{2,1}$ and $\rho_{2,0} >-1$. Then for any $f \in C_c^\infty(M)$,
$$\| \sfGamma^{\tensorh^*}( P_t f) \|_{L^1} \leq e^{- \alpha t} \| \sfGamma^{\tensorh^*}(f)\|_{L^1}, \quad \alpha := \frac{\rho_{2,0} \rho_1 + \rho_{2,1}}{\rho_{2,0} +1}.$$
Furthermore, if $\alpha > 0$ and $\tensorh^* + \tensorv^*$ is a complete Riemannian co-metric, then $\vol(M) < \infty$.
\item Assume that the conditions in \emph{(b)} hold with $\alpha
  >0$ and $\vol(M) < \infty$. Then for any $f \in C_c^\infty(M)$,
$$\| f- f_M \|_{L^2}^2 \leq \frac{1}{\alpha} \int_M \sfGamma^{\tensorh^*}(f) \dvol,$$
where $f_M = {\vol(M)}^{-1} \int_M f \dvol$.
As a consequence, if $\lambda$ is any non-zero eigenvalue of the
Friedrichs extension of~$L$, then $\alpha \leq - \lambda$.
\end{enumerate}
\end{proposition}

\begin{proof}
  \begin{enumerate}[(a)]
  \item By Proposition~\ref{prop:GradBound}~(a), we have $$\|
    \sfGamma^{\tensorh^*}(P_t f)\|_{L^1} \leq e^{- \alpha(\ell) t} \|
    \sfGamma^{\tensorh^* + \ell \tensorv^*}(f) \|_{L^1} $$ with
    $\alpha(\ell) = \min \{ \rho_1 - 1/\ell, \rho_{2,1} + \rho_{2,0}
    /\ell \}$ holds for any $f \in C_c^\infty(M)$. It follows that $\|
    \sfGamma^{\tensorh^*}(P_t f)\|_{L^1} \leq e^{- \alpha(\ell) t} \|
    \sfGamma^{\tensorh^* }(f) \|_{L^1}$ from
    Lemma~\ref{lemma:RemoveV}. For every~$t$, we then take the infimum
    over~$\ell$ to get
$$ \inf_\ell e^{- \alpha t} \leq e^{-kt}  \quad\text{with } k = \min\{ \rho_1, \rho_{2,1} \}.$$
\item With $\alpha(\ell)$ defined as in the proof of (a), note that if
  $\rho_1 \geq \rho_{2,1}$ and if $\rho_2 > -1$, then
$$\inf_\ell e^{- \alpha(\ell) t} 
= \exp\left(- \frac{\rho_{2,0} \rho_1 + \rho_{2,1}}{\rho_{2,0} +1} t\right) = e^{-\alpha t}$$
which gives us the first part of the result.

For the second part, we assume that $\rho_1 > \rho_{2,1}$, since if
$\alpha >0$ with $\rho_1 = \rho_{2,1}$, then we can always decrease
$\rho_{2,1}$ while keeping $\alpha$ positive. For two compactly supported
functions $f,g \in C_c^\infty(M)$, note that
\begin{align*}  \int_M &(P_t f-f)g \dvol = \int_M \int_0^t  \left(\frac{d}{ds}  P_sf \right) g \, ds \dvol \\
  &=  \frac{1}{2} \int_0^t \int_M   (\srL P_sf)  g \dvol ds
   = \frac{1}{2} \int_0^t \int_M \sfGamma^{\tensorh^*}(P_sf , g)
  \dvol ds.
\end{align*}
Hence, by the Cauchy-Schwartz inequality
\begin{align*} \left|\int_M (P_t f-f)g \dvol \right| \leq \frac{1}{2}
  \int_0^t \int_M \|\sfGamma^{\tensorh^*}(P_sf)\|_{L^\infty}^{1/2}
  \sfGamma^{\tensorh^*}(g)^{1/2} \, d\vol,
\end{align*}
which has upper bound
$$ \frac{1}{2} \left\|\sfGamma^{\tensorh^*}(f)+ \frac{\rho_{2,0} +1}{\rho_1 - \rho_{2,1}}\sfGamma^{\tensorv^*}(f) \right\|_{L^\infty}^{1/2} \int_M \sfGamma^{\tensorh^*}(g)^{1/2} \dvol \int_0^t e^{-\alpha s} ds,$$
by Proposition~\ref{prop:GradBound}~(a). From the spectral theorem, we
know that $P_tf$ reaches an equilibrium $P_\infty f$ which is in
$\Dom(L)$ and satisfies $L P_\infty f = 0$. Since this implies
$\sfGamma^{\tensorh^*}(P_\infty f) = 0$, we must have that $P_\infty f$ is a
constant.

Assume that $\vol(M) = \infty$. Then $P_\infty f = 0$ and hence, for
any $f,g \in C_c^\infty(M)$, we have
$$\left| \int_M fg \dvol \right| \leq \frac{1}{2 \alpha} \left\|\sfGamma^{\tensorh^*}(f) + \frac{\rho_{2,0} +1}{\rho_1 - \rho_{2,1}}\sfGamma^{\tensorv^*}(f) \right\|_{L^\infty}^{1/2} \int_M \sfGamma^{\tensorh^*}(g)^{1/2} \dvol .$$
However, since $\tensorg$ is complete, we can find a sequence of
functions $f_n \in C^\infty_c(M)$ such that $f_n \uparrow 1$ while
$\| \sfGamma^{\tensorg^*}(f_n)\|_{L^\infty} \to 0$. Inserting such a sequence for $f$
in the above formula and letting $n \to \infty$, we obtain the
contradiction that $\int_M g \dvol = 0$ for any $g \in C_c^\infty(M)$.
\item Follows from the identity
  \begin{align*} \|f- f_M\|_{L^2}^2
    &= \int_M f^2 \dvol - \frac{1}{\vol(M)} \left(\int_M f \dvol \right)^2\\
    &= - \int_0^\infty \frac{\partial}{\partial t} \int_M (P_t f)^2 \dvol dt\\
    &=  \int_0^\infty \int_M \sfGamma^{\tensorh^*}(P_tf) \dvol dt \leq
    \frac{1}{\alpha}
    \|\sfGamma^{\tensorh^*}(f)\|_{L^1}.\end{align*}\qed
\end{enumerate}
\end{proof}

\section{Entropy and bounds on the heat
  kernel} \label{sec:GeodesicFibers}
\subsection{Commutating condition on $\sfGamma^{\tensorh^*}$ and
  $\sfGamma^{\tensorv^*}$} \label{sec:CommConditions} For some of our
inequalities involving logarithms, we will need the following
condition. Let $L \in \Gamma(T^2M)$ be a second order operator without
constant term with positive semi-definite $\tensorq_L = \tensorh^*$
defined as in \eqref{qform}. Assume that $L$ satisfies either
\eqref{CDstar} or \eqref{CD} with respect to positive
semi-definite~$\tensorv^*$. We say that condition \eqref{CondB} holds if
\begin{equation}
  \label{CondB} \tag{\sf B} \sfGamma^{\tensorh^*}(f, \sfGamma^{\tensorv^*}(f)) = \sfGamma^{\tensorv^*}(f, \sfGamma^{\tensorh^*}(f)) \qquad \text{ for every } f \in C^\infty(M).
\end{equation}
We make the following observation.
\begin{lemma} \label{lemma:iff}
  Let $\tensorg$ be a Riemannian metric on a manifold $M$, with an orthogonal splitting $TM =
  \calH \oplus_{\perp} \calV$ and use this decomposition to define the
  connection $\rnabla$ as in~\eqref{rnabla}. Write $\tensorg |_{\calH} = \tensorh$ and
  $\tensorg|_{\calV} = \tensorv$ and let $\tensorh^*$ and $\tensorv^*$
  be their respective corresponding co-metrics. Then
$$\sfGamma^{\tensorh^*}(f, \sfGamma^{\tensorv^*}(f)) = \sfGamma^{\tensorv^*}(f, \sfGamma^{\tensorh^*}(f))$$
holds for every $f \in C^\infty(M)$ if and only if $\rnabla \tensorv^*
= \rnabla \tensorh^* = 0$. 
\end{lemma}
\begin{proof} It is simple to verify that for any $A \in \Gamma(\calH)$ and $V \in \Gamma(\calV)$, we have
$$\rnabla_A \tensorh^* = 0, \qquad \rnabla_V \tensorv^* = 0, \qquad T^{\rnabla}(A,V) = 0,$$
where $T^{\rnabla}$ is the torsion of $\rnabla$. Define $\shh$ as in Section~\ref{sec:Summary} and let $\shv$ be defined analogously. Using the properties of $\rnabla$, we get
  \begin{align*}
    \sfGamma^{\tensorh^{*}}(f, \sfGamma^{\tensorv^*}(f)) &-
    \sfGamma^{\tensorv^*}(f, \sfGamma^{\tensorh^*}(f))
    = (\shh df) \|df \|_{\tensorv^*}^2 - (\shv df) \| df \|_{\tensorh^*}^2 \\
    &=  2 \rnabla_{\shh df} df(\shv df) - 2 \rnabla_{\shv df} df (\shh df) \\
    &\quad + (\rnabla_{\shh df} \tensorv^*)(df,df) -  (\rnabla_{\shv df} \tensorh^*)(df,df) \\
    &= (\rnabla_{\shh df} \tensorv^*)(df,df) - (\rnabla_{\shv df}
    \tensorh^*)(df,df).
  \end{align*}
  Since $T^*M = \ker \tensorh^{*} \oplus \ker \tensorv^*$ and since
  $\rnabla$ preserves these kernels, the above expression can only
  vanish for all $f \in C^\infty(M)$ if $\rnabla \tensorh^* = 0$ and
  $\rnabla \tensorv^*=0$.\qed
\end{proof}

Let $L$, $P_t$ and $X(x)$ be as in
Section~\ref{sec:GeneralOperator}. In this section, we explore the
results we obtain when both conditions \eqref{CondA} and \eqref{CondB} hold.
We will also assume that $L$ satisfies
\eqref{CD} rather than \eqref{CDstar}. The reason for this is that in
the concrete case when $L$ is the sub-Laplacian of a sub-Riemannian
manifold with an integrable metric-preserving complement, the
condition \eqref{CondB} along with the assumptions of
Theorem~\ref{th:CD} imply \eqref{CD}, see
Section~\ref{sec:Summary}. For most of the results, we also need the
requirement that $\rho_2 >0$. This means that we can use the results
of \cite{BaGa12,BaBo12,BBG14}.

Let us first establish some necessary identities. Let $P_t$ be the
minimal semigroup of $\frac{1}{2}L$ where $\tensorq_L =
\tensorh^*$. For a given $T >0$, let $u_t := P_{T-t}f$ with $f \in C^\infty(M) \cap L^\infty(M)$.
It is
clear that $(\frac{1}{2}L+ \frac{\partial}{\partial
  t})\sfGamma^{\tensors^*}(u_t) = \sfGamma^{\tensors^*}_2(u_t)$ for
any $\tensors^* \in \Gamma(\Sym^2 TM)$. Also note that if $F\colon U
\subseteq \real \to \real$ be a smooth function, then for any $f \in
C^\infty(M)$ with values in $U$, we obtain
\begin{align*} L F(f) =  F'(f) Lf + F''(f)
  \sfGamma^{\tensorh^*}(f).\end{align*}

Straight-forward calculations lead to the following identities.
\begin{lemma} \label{lemma:partialtL}\quad
  \begin{enumerate}[\rm(a)]
  \item If $u_t = P_{T-t} f$ has values in the domain of $F$, then $$\left(
      \frac{1}{2} L + \frac{\partial}{\partial t} \right) F(u_t) =
    \frac{1}{2} F''(u_t) \sfGamma^{\tensorh^*}(u_t).$$ In particular, if
    $u_t$ is positive then
    \begin{align*}
      &\left(\frac{1}{2} L + \frac{\partial}{\partial t} \right) \log u_t =  - \frac{\sfGamma^{\tensorh^*}(u_t)}{2u^2_t},  \\
      &\left(\frac{1}{2} L + \frac{\partial}{\partial t} \right) u_t\log
      u_t = \frac{\sfGamma^{\tensorh^*}(u_t)}{2u_t} = \frac{1}{2} u_t
      \sfGamma^{\tensorh^*}(\log u_t).
    \end{align*}
  \item For any $\tensors^* \in \Gamma(\Sym^2 TM)$, we have
    \begin{align*}
    &\left(\frac{1}{2} L + \frac{\partial}{\partial t}\right) u_t\sfGamma^{\tensors^*}(\log u_t) \\
    &\quad= u_t\, \sfGamma^{\tensors^*}_2(\log u_t)+ u_t \left(
        \sfGamma^{\tensorh^*}(\log u_t,\sfGamma^{\tensors^*}(\log u_t)) -
        \sfGamma^{\tensors^*}(\log u_t,\sfGamma^{\tensorh^*}(\log u_t))
      \right).
    \end{align*}
  \end{enumerate}
\end{lemma}
In particular, $\left(\frac{1}{2}L + \frac{\partial}{\partial t}
\right) u_t \sfGamma^{\tensorh^*}(\log u_t) = u_t
\sfGamma_2^{\tensorh^*}(\log u_t).$ If $\tensorv^*$ is any co-metric
such that $\sfGamma^{\tensorh^*}(f, \sfGamma^{\tensorv^*}(f)) = \sfGamma^{\tensorv^*}(f, \sfGamma^{\tensorh^*}(f)),$
then $\left(\frac{1}{2}L + \frac{\partial}{\partial t} \right) u_t
\sfGamma^{\tensorv^*}(\log u_t) = u_t \sfGamma_2^{\tensorv^*}(\log u_t)$ as
well.

\subsection{Entropy bounds and Li-Yau type inequality} \label{sec:ELY}
We follow the approach of \cite{ArTh10}, \cite[Theorem~5.2]{BaGa12} and \cite[Theorem~1.1]{Wan13}.

\begin{lemma} \label{lemma:ABLambda}  Assume that $L$ satisfies
  \eqref{CD}. Also assume that \eqref{CondA} and \eqref{CondB}
  hold. Consider three continuous functions $a, b, \ell: [0,T] \to
  \real$ with $a(t)$ and $\ell(t)$ being non-negative. Let $C$ be a
  constant. Assume that $a(t), b(t)$ and $\ell(t)$ are smooth for $t
  \in (0,T)$ and on the same domain satisfy
  \begin{equation} \label{ABLambda} \left\{ \begin{array}{l}
        0 \leq \dot a(t) + \left(\rho_1 - \frac{1}{\ell(t)}  - 2 b(t) \right) a(t) + C \\
        0 \leq \dot \ell(t) + \rho_2 + \frac{\dot a(t)}{a(t)}
        \ell(t).  \end{array} \right.
 \end{equation}
Consider a positive function $f \in C^\infty_b(M)$, $f > 0$ with bounded gradient 
$\sfGamma^{\tensorh^* + \tensorv^*}(f) \in C_b^\infty(M)$. Then we have
  \begin{align*}  a(0)\, P_T f\, & \sfGamma^{\tensorh^* + \ell(0) \tensorv^*}(\log P_T f ) - a(T) P_T\left( f \sfGamma^{\tensorh^* + \ell(T) \tensorv^*}(\log f)\right)  \\
    &\leq  2C \left(P_T(f \log f ) - (P_T f) \log P_T f \right)  \\
    & + n \left(\int_0^T a(t) b(t)^2 dt \right) P_T f -  2 \left(\int_0^T a(t) b(t) dt\right) P_T L f. \end{align*}
\end{lemma}

\begin{proof}
We have $P_t f > 0$ from our assumptions on $L$ and $f$. 
  For any $T > 0$, define $u_t = P_{T-t} f$ for $0 \leq t \leq T$ and
  \begin{align*} Y_t &=  a(t) \left( u_{t}\, \sfGamma^{\tensorh^*}( \log u_{t}) + \ell(t) u_t \,\sfGamma^{\tensorv^*}(\log u_{t}) \right) \circ X_t \\
    &\quad + 2 C \left(u_{t} \log u_{t} \right) \circ X_t +
    \int_0^t a(s) \left(n b(s)^2 u_{s} - 2 b(s) L u_{s} \right) \circ
    X_s \, ds. \end{align*}
Let us write $\equivM$ for equivalence modulo differentials of local
  martingales.  We use that $$L u_{t} = u_t L \log u_{t} +
  \frac{\sfGamma^{\tensorh^*}(u_{t})}{u_{t}} = u_t L \log u_{t} +
  u_t\sfGamma^{\tensorh^*}(\log u_{t})$$ and \eqref{CD} to obtain
  \begin{align*}
    dY_t \equivM & \left(\dot a(t)- 2 a(t) b(t)  + C\right) u_t \sfGamma^{\tensorh^*}(\log u_t) \circ X_t \, dt \\
    & + \left(\dot a(t) \ell(t) + a(t) \dot \ell(t) \right) \sfGamma^{\tensorv^*}(\log u_t) \circ X_t \, dt \\
    & + a(t) u_t \sfGamma^{\tensorh^*+ \ell(t) \tensorv^*}_2(\log u_t) \circ X_t \, dt \\
    & + a(t) u_{t} \left(n b(t)^2 - 2 b(t) L \log u_{t} \right) \circ X_t \, dt \\
    \geq &   \left(\dot a(t) + \left(\rho_1  - \frac{1}{\ell(t)} - 2 b(t) \right) a(t)  +C \right) u_t \sfGamma^{\tensorh^*}(\log u_t) \circ X_t \, dt \\
    & + a(t) \left(\dot \ell(t) + \rho_2 + \frac{\dot a(t)}{a(t)} \ell(t) \right) \sfGamma^{\tensorv^*}(\log u_t) \circ X_t \, dt \\
    & + n a(t) u_{t} \left(b(t) - L \log u_{t} \right)^2 \circ X_t \, dt.
  \end{align*}
$Y$ is then a submartingale from \eqref{ABLambda}. The result
  follows from $\expect[Y_T] \geq \expect[Y_0]$.\qed
\end{proof}

We look at some of the consequences of Lemma~\ref{lemma:ABLambda}.
\begin{corollary} \label{cor:EntropyLY} Assume that $L$ satisfies
  \eqref{CD} with $\rho_2 >0$, and that \eqref{CondA} and
  \eqref{CondB} also hold. Let $f \in C^\infty_b(M) $ be any
  bounded smooth function with $\sfGamma^{\tensorh^* + \tensorv^*}(f) \in C_b^\infty(M)$.
  \begin{enumerate}[\rm(a)]
  \item \emph{(Entropy bound)} Assume that $\rho_1 \geq 0$ and that $f > 0$. Then for any $x \in M$,
    \begin{align*} \frac{1- e^{-\rho_1 t}}{2 \rho_1} \sfGamma^{\tensorh^* }(\log P_t f)(x) \leq &
      \left(1 + \frac{2}{\rho_2}  \right) P_t \left(
        \frac{f}{P_t f(x)} \log \frac{f}{P_t f (x)}
      \right)(x). \end{align*} 
  \item \emph{(Li-Yau inequality)} Assume that $n < \infty$ in
    \eqref{CD} and that $f \geq 0$, not identically zero. Then for any $1 < \beta < 2$ and for any $t \geq 0$,
    \begin{align} \label{LY} \frac{ \sfGamma^{\tensorh^* }(P_t
        f)}{(P_t f)^2} - \left(a_\beta - b_\beta \rho_1 t
      \right) \frac{P_t L f}{P_t f} \leq \frac{n}{4 t} \left(
        \frac{a_\beta^2}{(2-\beta)(\beta-1)} - \rho_1 t (2
        a_\beta - b_\beta \rho_1 t) \right) \end{align} where
    $\displaystyle a_\beta = \frac{\rho_2 + \beta}{\rho_2}$ and
    $\displaystyle b_\beta = \frac{\beta-1}{\beta}$.
  \end{enumerate}
\end{corollary}

The special case of $\beta = {2}/{3}$ in \eqref{LY} is described with
consequences in \cite[Theorem~6.1]{BaGa12}. If $\rho_1 \geq 0$, then for
many application $\beta = \sqrt{(2 + \rho_2)(1+ \rho_2)} - \rho_2$ is
a better choice, as this minimizes the ratio of ${a_\beta^2}/{(2
  -\beta)(\beta-1)}$ over $a_\beta$. With this choice, we obtain
relation
\begin{align} \label{LY2} \frac{1}{D} \frac{ \sfGamma^{\tensorh^* }(P_t f)}{(P_t f)^2} - \frac{P_t
    L f}{P_t f} \leq \frac{N}{t} , \end{align} where
\begin{equation} \label{Ndim} 
N := \frac{n}{4} \frac{(\sqrt{2+\rho_2}
    + \sqrt{1+\rho_2})^2}{\rho_2}, \quad D = \frac{ \sqrt{(2+ \rho_2)(1+\rho_2)}}{\rho_2} .
\end{equation}

\begin{proof} Recall that if $f \in C^\infty_b(M)$ is non-negative and non-zero, then $P_tf$ is strictly positive.
  \begin{enumerate}[(a)]
  \item We will use Lemma~\ref{lemma:ABLambda}. As in
    Proposition~\ref{prop:GradBound}~(c), for any $T \geq 0$, define $$a(t)
    = \frac{1 -e^{-\rho_1(T-t)}}{\rho_1},\quad \ell(t) = \rho_{2,0} \frac{\int_t^T a(s) \, ds }{a(t)} = \rho_{2,0} \frac{e^{-\rho_1(T-t)} -1 + \rho_1 (T-t)}{\rho_1 (1- e^{-\rho_1 (T-t)})} $$ and $C = 1 + {2}/{\rho_2}$. If we define $b(t) \equiv
    0$, condition \eqref{ABLambda} is satisfied. Hence, 
    \begin{align*} \frac{1 - e^{-\rho_1T}}{\rho_1} \frac{ \sfGamma^{\tensorh^* + \frac{\rho_{2} T}{2}
          \tensorv^*}(P_T f)}{P_T f} \leq \left( 1 +
        \frac{2}{\rho_2} \right) \big(P_T(f \log f) - (P_T f)
        \log P_T f \big).  \end{align*} Divide by $P_Tf$ and
    evaluate at $x$ for the result.
  \item For any $\ve > 0$, define $f_{\ve} = f + \ve > 0$. For any $\alpha >0$ and $T >0$, define $\ell(t) =
    \frac{\rho_2}{\alpha +2} (T - t)$, $a(t) = (T-t)^{\alpha+1}$ and
$$b(t) = \frac{1}{2} \left(\rho_1 +\frac{\dot a}{a} - \frac{1}{\ell} \right) = \frac{1}{2} \left(\rho_1 -\left( \alpha+1 + \frac{\alpha+2}{\rho_2} \right) \frac{1}{T-t} \right) .$$
Note that
\begin{align*} \int_0^T a(t) b(t) \, dt &= \frac{1}{2} \left( \frac{\rho_1}{\alpha+2} T^{\alpha+2} - \left(1 + \frac{\alpha+2}{\rho_2(\alpha +1)} \right) T^{\alpha+1} \right),\\
  \int_0^T a(t) b(t)^2 \, dt &=  \frac{1}{4} \left( \frac{\rho_1^2}{\alpha+2} T^{\alpha+2} \right.  \left. - 2 \rho_1 \left(1 + \frac{\alpha+2}{\rho_2(\alpha +1)} \right) T^{\alpha+1} \right. \\
  & \qquad+\left.  \frac{(\alpha+1)^2}{\alpha} \left(1 +
      \frac{\alpha+2}{\rho_2(\alpha +1)} \right)^2 T^{\alpha} \right).
\end{align*}
If we put $C = 0$, then \eqref{ABLambda} is satisfied and so if we use
$f_{\ve}$ in Lemma~\ref{lemma:ABLambda} and let $\ve \downarrow 0$, we get
\begin{align*} & \frac{ \sfGamma^{\tensorh^* + \frac{\rho_2T
      }{\alpha+2} \tensorv^*}(P_T f)}{P_T f} +
  \left(\frac{\rho_1}{\alpha+2} T- \left(1 +
      \frac{\alpha+2}{\rho_2(\alpha +1)} \right) \right) P_T L f\\
  \nonumber &\leq \frac{n}{4} \left(\frac{\rho_1^2}{\alpha+2} T - 2
    \rho_1 \left(1 + \frac{\alpha+2}{\rho_2(\alpha +1)} \right) +
    \frac{(\alpha+1)^2}{\alpha} \left(1 +
      \frac{\alpha+2}{\rho_2(\alpha +1)} \right)^2 \frac{1}{T} \right)
  P_T f. \end{align*} Define $\beta := (\alpha +2)/(\alpha +1)$ to
obtain \eqref{LY}.\qed
\end{enumerate}
\end{proof}

Using \eqref{LY2} and the approach found in \cite[Remark 6.2 and Section
7]{BaGa12} and \cite{BBG14}, we obtain the following results.
\begin{corollary} Assume that $L$ satisfies \eqref{CD} relative to
  $\tensorv^*$ with $\rho_1 \geq 0, \rho_2 > 0$ and $n <
  \infty$. Write $\tensorg^* = \tensorh^* + \tensorv^*$. Also assume
  that \eqref{CondA} and \eqref{CondB} hold and that $L$ is symmetric
  with respect to the volume form $\vol$. Let $p_t(x, y)$ be the
  heat kernel of $\frac{1}{2} L$ with respect to $\vol$. Finally, let
  $N$ and $D$ be as in \eqref{Ndim}. Then the following holds.
  \begin{enumerate}[\rm(a)]
  \item $p_t(x,x) \leq t^{-N/2} p_1(x,x)$ for any $x \in M$.
  \item For any $0 < t_0 < t_1$ and any $f \in C^\infty_b(M) $
    non-negative, not identically zero,
    \begin{equation} \label{ParabolHarnack} P_{t_0} f(x) \leq
      (P_{t_1}f)(y) \left(\frac{t_1}{t_0}\right)^{N/2} \exp\left(D\,
        \frac{\metricd_{cc}(x,y)^2}{2  (t_1-t_0)}
         \right)\end{equation} where
    $\metricd_{cc}$ is the Carnot-Carath\'eodory distance. If
    $\tensorg^*$ is the co-metric of a complete Riemannian metric,
    then
$$p_{t_0} (x,y) \leq p_{t_1}(x, z) \left(\frac{t_1}{t_0}\right)^{N/2} \exp\left(D \,\frac{\metricd_{cc}(y,z)^2}{2 (t_1-t_0)}  \right).$$
\end{enumerate}
\end{corollary}

There are several more results which we can obtain when \eqref{CondA}
and \eqref{CondB} hold, along with the fact that $L$ satisfies
\eqref{CD} with $\rho_2 > 0$, which can be found in
\cite{BaGa12,BaBo12,BBG14}. We list some of the most important results
here, found in \cite[Theorem~10.1]{BaGa12} and \cite[Theorem~1.5]{BBG14}.

\begin{theorem} \label{th:others} Let $L$ be a second order operator satisfying \eqref{CD}
  with respect to $\tensorv^*$ and with $\rho_2 >0$. Assume that it is
  symmetric with respect to some volume form $\vol$. Define
  $\tensorg^* = \tensorh^* + \tensorv^*$ and assume that this is a
  complete Riemannian metric. Finally, assume that conditions
  \eqref{CondA} and \eqref{CondB} hold.
Let $B_r(x)$ be the ball of radius $r$ centered at $x \in M$ with
  respect to the metric~$\metricd_{cc}$.
  \begin{enumerate}[\rm(a)]
  \item \emph{(Sub-Riemannian Bonnet-Myers Theorem)} If $\rho_1 > 0$,
    then $M$ is compact.
  \item \emph{(Volume doubling property)} If $\rho_1 \geq 0$, there
    exist a constant $C$ such that
$$\vol(B_{2r}(x)) \leq C \vol(B_r(x)), \quad \text{for any } r \geq 0.$$
\item \emph{(Poincar\'e inequality on metric-balls)} If $\rho_1 \geq
  0$, there exist a constant $C$ such that
$$\int_{B_r(x)} \|f - f_{B_r}\|^2 \dvol \leq C r^2 \int_{B_r(x)} \sfGamma^{\tensorh^*} (f) d\vol,$$
for any $r\geq 0$ and $f \in C^1\left( \bar B_r(x)\right)$ 
where $\displaystyle f_{B_r} = \vol(B_r(x))^{-1} \int_{B_r(x)} f \dvol$.
\end{enumerate}
\end{theorem}

\section{Examples and comments} \label{sec:Comments}
\subsection{Results in the case of totally geodesic Riemannian foliations} \label{sec:sRSpecial}
Let us consider the following case. Let $(M, \tensorg)$ be a
Riemannian manifold, and let $\calH$ be a subbundle that is bracket
generating of step 2, i.e. the tangent bundle is spanned by the
sections of $\calH$ and their first order brackets. Let $\calV$
be the orthogonal complement of $\calH$ with respect to $\tensorg$.
Define $\rnabla$ with respect to the decomposition $TM = \calH \oplus
\calV$ and let $\tensorh$ and $\tensorv$ be the respective restrictions
of $\tensorg$ to $\calH$ and $\calV$. Let us make the following
assumptions:
\begin{enumerate}[\ -]
\item $\calV$ is integrable, $\tensorg_{\mathstrut}$ is complete, $\rnabla \tensorg = 0$
and the assumptions (i)--(iv) of Section~\ref{sec:Summary} hold with $\mcalR > 0$.
\end{enumerate}
From our investigations so far, we then know that
\begin{enumerate}[\ -]
\item $\calV$ is a metric-preserving complement of $(M, \calH,
  \tensorh)$; the foliation of $\calV$ is a totally geodesic Riemannian foliation.
\item the sub-Laplacian $\srL$ of $\calV$ is symmetric with respect to
  the volume form $\vol$~of~$\tensorg$;
\item $\srL$ is essentially self-adjoint on $C_c^\infty(M)$;
\item $\srL$ satisfies \eqref{CD} with respect to $\tensorv^*$;
\item both \eqref{CondA} and \eqref{CondB} hold.
\end{enumerate}
We list the results that can be deduced on such manifolds using
the approach of the generalized curvature-dimension inequality. We
will split the results up into two propositions.

\begin{proposition} \label{prop:pzcknn} Define $\kappa = \frac{1}{2}
  \mcalRvar^2 \rRicH - \MRicHV^2$ and assume that $\kappa
  \geq 0$. Let $f \in C_b^\infty(M)$ be non-negative, not identically zero. Define
  $$N = \frac{n}{4}
\frac{ \left(\sqrt{2 \rRicH + \kappa} + \sqrt{\rRicH+ \kappa   }\right)^2}{\kappa} , \qquad D = \frac{\sqrt{(\kappa+ \rRicH)(\kappa  + 2\rRicH)}}{ \kappa}.$$
  \begin{enumerate}[\rm(a)]
  \item Assume that $\sfGamma^{\tensorh^* + \tensorv^*}(f) \in C^\infty_b(M)$. Then for any $1< \beta < 2$, we have
    \begin{align*} \frac{ \sfGamma^{\tensorh^* }(P_t f)}{(P_t f)^2} -
      \left(1 + \frac{\rRicH}{2\kappa} \beta \right) \frac{P_t L f}{P_t
        f} \leq \frac{n}{4 t} \left( \frac{\big(1 + \frac{\rRicH}{2
              \kappa} \beta \big)^2}{(2 -\beta)(\beta-1)}
      \right). \end{align*}
  \item Let $p_t(x, y)$ be the heat kernel of $\frac{1}{2} \srL$
    with respect to $\vol$. Then
$$p_t(x,x) \leq \frac{1}{t^{N/2}} p_1(x,x)$$
for any $x \in M$ and $0 \leq t \leq 1$. Furthermore, for any $0 < t_0 < t_1$,
$$P_{t_0} f(x) \leq (P_{t_1}f)(y) \left(\frac{t_1}{t_0}\right)^{N/2} \exp\left(D\,\frac{\metricd_{\mathrm{cc}}(x,y)^2}{2 (t_1-t_0)} \right).$$
\end{enumerate}
In both results, if $\kappa =0$, we interpret the quotient
${\kappa}/{\rRicH}$ as $\frac{1}{2} \mcalRvar^2$.
\end{proposition}
Note that if~$\MRicHV = 0$, the constant in the above result
is independent of~$\rRicH$.

\begin{proof}
  From the formulas \eqref{rhoSR2}, we know that $\srL$ satisfies
  \eqref{CD} with $\rho_2 > 0$ and $\rho_1 \geq 0$. In particular, we
  can choose $c ={1}/{\rRicH}$ if $\rRicH > 0$ and $\infty$ if
  $\rRicH =0$. This choice gives us $\rho_1 = 0$, while maximizing
  $\rho_2$. Note that if $\rRicH =0$, then $\MRicHV$ must be $0$ as well, since
  we have required $\kappa \geq 0$.\qed
\end{proof}

\begin{example}[Free nilpotent Lie algebra of step 2]
Let $\frakh$ be a vector space of dimension $n$ with an inner product $\langle \newbullet, \newbullet \rangle$ and let $\frakk$ denote the vector space $\bigwedge^2 \frakh$. Define a Lie algebra $\frakg$ as the vector space $\frakh \oplus \frakk$ with Lie brackets determined by $\frakk$ being the center and for any $A,B \in \frakh$, we have
$$[A, B] = A \wedge B \in \frakk,.$$
This is clearly a nilpotent Lie algebra of step 2 and dimension~${n(n+1)}/{2}$.

Let $G$ be a simply connected nilpotent Lie group with Lie algebra $\frakg$. Define a sub-Riemannian structure $(\calH, \tensorh)$ by left translation of $\frakh$ and its inner product. Let $A_1, \dots, A_n$ be a left
  invariant orthonormal basis of $\calH$ and define $L = \sum_{i=1}^n
  A_i^2$. From Part~I, Example~4.4, we know that $L$ satisfies
  \eqref{CD} with respect to some $\tensorv^*$, $n = \rank \frakh$, $\rho_1 = 0$ and
  $\rho_2 = \frac{1}{2(n-1)}$. This choice of $\tensorv^*$ also gives us a complete Riemannian metric~$\tensorg$ satisfying $\rnabla \tensorg =0$ and with $L$ being the sub-Laplacian of the volume form of $\tensorg$. 
We then obtain that for any $0 < t_0 < t_1$ and $f \in C_b^\infty(M)$
$$P_{t_0} f(x) \leq (P_{t_1}f)(y) \left(\frac{t_1}{t_0}\right)^{N/2} \exp\left(D\,\frac{\metricd_{\mathrm{cc}}(x,y)^2 }{2 (t_1-t_0)}  \right)$$
where $N = \frac{n}{4} \left(\sqrt{4 n - 3 } + \sqrt{2 n- 1
  }\right)^2$ and $D = \sqrt{(2n-1)(4n-3)}$.
\end{example}\goodbreak

\begin{proposition} \label{prop:pzckp} Define $\kappa = \frac{1}{2}
  \mcalRvar^2 \rRicH - \MRicHV^2$ and assume that $\kappa >
  0$. Then the following statements hold.
  \begin{enumerate}[\rm(a)]
  \item $M$ is compact.
  \item If $f \in C^\infty(M)$ is an arbitrary function and $$\displaystyle\alpha := \left( \frac{2 \kappa}{2
        \MRicHV+\mcalR \sqrt{2 \rRicH + 2\kappa }} \right)^2,$$ we have
    \begin{align*}
      \|\sfGamma^{\tensorh^*}(P_t f)\|_{L^1}  \leq  e^{-\alpha t} \|\sfGamma^{\tensorh^*}(f) \|_{L^1}, \quad
 \text{and}   \quad  \| f- f_M \|_{L^2}^2 \leq \frac{1}{\alpha} \int_M
      \sfGamma^{\tensorh^*}(f) \dvol
    \end{align*}
    where $f_M = {\vol(M)}^{-1} \int_M f \dvol$.
  \item Let $f \in C^\infty(M)$ be an arbitrary function. Then
    \begin{align*} t \sfGamma^{ \tensorh^*}(P_t f) \leq
      \left(1 + \frac{2 \rRicH}{\kappa} \right) (P_t f^2 - (P_t
      f)^2).\end{align*}
  \item Let $f$ be a strictly positive smooth function. Then for any $x \in M$,
    \begin{align*} t \sfGamma^{ \tensorh^*}(\log P_t f)(x) \leq
      2\left(1 + \frac{2\rRicH}{\kappa} \right) P_t \left( \frac{f}{f(x)}
        \log \frac{f}{f (x)} \right)(x).
    \end{align*}
  \end{enumerate}
\end{proposition}

\begin{proof}
  From the formulas \eqref{rhoSR2}, we know that $\srL$ satisfies
  \eqref{CD} with $\rho_2 > 0$ and $\rho_1 > 0$.
  \begin{enumerate}[(a)]
  \item Follows directly from Theorem~\eqref{th:others}~(a).
  \item We use Propositions~\ref{prop:GradBound}~(a) and
    \ref{prop:Poincare}~(b). With our assumption of
    $\rnabla \tensorv =0$, the formulas \eqref{rhoSR} show that we
    can choose $\rho_{2,1} =0$ and both $\rho_1$ and $\rho_2$ strictly
    positive, since $\kappa >0$. The result follows by maximizing
    $\frac{\rho_{2} \rho_1}{\rho_2 + 1}$ with respect to $c$.
  \item We use Proposition~\ref{prop:GradBound}~(c) and using \eqref{rhoSR2} with $c = {1}/{\rRicH}$.
\item Similar to the proof of (c), only using
  Corollary~\ref{cor:EntropyLY}~(a) instead.\qed
\end{enumerate}
\end{proof}

\begin{example}
Let $\frakg$ be a compact semisimple Lie algebra with bi-invariant metric
$$\langle A,B \rangle = - \frac{1}{4\rho} \tr\, \ad(A)\ad(B), \quad \rho >0.$$
Let $G$ be a (compact) Lie group with Lie algebra $G$ and with metric $\widecheck{\tensorg}$ given by left (or right) translation of the above inner product. Then $\rho >0$ is the lower Ricci bound of $G$.

Let $\frakh$ be the subspace of the Lie algebra $\frakg \oplus \frakg$ consisting of elements on the form $(A, 2A)$, $A \in \frakg$. Define the subbundle $\calH$ on $G \times G$ by left translation of $\frakh$. If we use the same symbol for an element in the Lie algebra and the corresponding left invariant vector field, we define a metric $\tensorh$ on $\calH$ by
$$\tensorh((A,2A), (A,2A) ) = \langle A, A \rangle.$$
Define $\pi: G \times G \to G$ as projection on the second coordinate with vertical bundle $\calV = \ker \pi_*$ and give this bundle a metric $\tensorv$ determined by
$$\|(A,0)\|_{\tensorv}^2 = \frac{1}{4\rho} \langle A, A \rangle.$$
If we then define $\rnabla$ relative to $\calH \oplus \calV$ and $\tensorg =\pr_{\calH}^* \tensorh + \pr_{\calV}^* \tensorv$, then $\rnabla \tensorg = 0$. Let $\srL$ be the sub-Laplacian with respect to
  $\calV$, which coincides with the sub-Laplacian of the volume form
  $\vol$ of $\tensorg$. We showed in
  Part~I, Example~4.6 that this satisfies \eqref{CD} with respect to
  $\tensorv^*$, $n = \dim G$, $\rho_1 = \rRicH = 4\rho$ and $\rho_2 = \frac{1}{2} \mcalRvar^2 = {1}/{4}$.

 We then have that for any $f \in C^\infty(G \times G)$,
$$\|f - f_{G \times G}\|^2_{L^2} = \frac{5}{4 \rho} \int_M \sfGamma^{\tensorh^*}(f) \dvol$$
where $f_{G \times G} = {\vol(G \times G)}^{-1} \int_{G \times G} f \dvol.$
\end{example}

\subsection{Comparison to Riemannian Ricci curvature} 
Let us consider a sub-Riemannian manifold such as in
Section~\ref{sec:sRSpecial}. Given the results of
Proposition~\ref{prop:pzcknn} and Proposition~\ref{prop:pzckp}, it
seems reasonable to consider sub-Riemannian manifolds with $\kappa \geq
0$ or $\kappa > 0$ as the analogue of Riemannian manifolds with
respectively non-negative and positive Ricci curvature. However, given
the extra structure in the choice of $\tensorv$ on $\calV$, it is
natural to ask how these sub-Riemannian results compare to the Ricci
curvature of the metric $\tensorg = \pr_{\calH}^* \tensorh +
\pr_{\calV}^* \tensorv$. We give the comparison here.

Introduce the following symmetric $2$-tensor
\begin{align*}
  \RicV(Y,Z) &= \tr \left(V \mapsto \pr_{\calV} R^{\rnabla}(V, Y) Z
  \right). \end{align*} Then the Ricci curvature of $\tensorg$ can be
written in the following way.
\begin{proposition} \label{prop:RiemannRicci} The Ricci curvature
  $\Ric_{\tensorg}$ of $\tensorg$ satisfies
  \begin{align} \label{RiemannRicci} \Ric_{\tensorg}(Y,Y) & =
    \RicH(Y,Y) + \RicHV(Y,Y) + \frac{1}{2} \|\tensorg(Y,
    \calR(\newbullet, \newbullet) )\|^2_{\wedge^2 \tensorg^*} \\
    \nonumber &\quad + \RicV(Y, Y) - \frac{3}{4} \|\calR(Y,\newbullet)\|^2_{\tensorg^* \otimes
      \tensorg}.
  \end{align}
\end{proposition}

Before we get to the proof, let us note the consequences of this
result. If $\kappa = \frac{1}{2} \rRicH \mcalRvar^2 - \MRicHV^2 $
is respectively non-negative or positive, this ensures that the first
line of \eqref{RiemannRicci} has respectively a non-negative or
positive lower bound. Furthermore, note that this part is independent of any covariant derivative of vertical vector fields.

\begin{proof}[of Proposition~\upshape\ref{prop:RiemannRicci}]
  Let $\nabla$ be the Levi-Civita connection of $\tensorg$. Define a
  two tensor $\calB(A,Z) = \nabla_A Z - \rnabla_A Z$.  Then it is clear
  that
  \begin{align*} R^{\nabla}(A,Y)Z &=  R^{\rnabla}(A,Y) Z + (\rnabla_A \calB)(Y,Z) - (\rnabla_Y \calB)(A,Z)  \\
    &\quad + \calB(\calB(Y,A), Z) + \calB(A, \calB(Y,Z)) - \calB(\calB(A,Y),Z) - \calB(Y,
    \calB(A,Z)).\end{align*} 
Furthermore, it is simple to verify that
  \begin{align*} \calB(A,Z) =& \frac{1}{2} \calR(A,Z) - \frac{1}{2} \sharp \tensorg(A, \calR(Z,\newbullet)) - \frac{1}{2} \sharp \tensorg(Z, \calR(A,\newbullet)) .
  \end{align*}
 Let $A_1, \dots, A_n$ and $V_1, \dots, V_\nu$ be local orthonormal
  bases of respectively $\calH$ and~$\calV$. Then
  \begin{align*} & \sum_{i=1}^n \tensorg(A_i, R^{\nabla}(A_i,Z)Z - R^{\rnabla}(A_i,Z) Z) \\
    &\quad = \sum_{i=1}^n \tensorg(Z, (\rnabla_{A_i}
    \calR)(A_i,Z)) - \frac{3}{4} \|\calR(Z, \newbullet) \|^2_{\tensorg^*
      \otimes \tensorg} + \frac{1}{2} \|\tensorg(Z,
    \calR(\newbullet, \newbullet))\|^2_{\wedge^2 \tensorg^*}.
  \end{align*}
  Similarly, $\sum_{s=1}^\nu
  \tensorg(R^{\nabla}(V_s,Z)Z - R^{\rnabla}(V_s,Z) Z, V_s) = 0$.\qed
\end{proof}

\subsection{Generalizations to equiregular submanifolds of steps greater than two}
Many of the results in Section~\ref{sec:GradientBound} and Section
\ref{sec:GeodesicFibers} depend on the condition $\rho_2 >0$. A
necessary condition for this to hold is that our sub-Riemannian
manifold is bracket-generating of step 2. Let us note some of the
difficulties in generalizing the approach of this paper to
sub-Riemannian manifolds $(M, \calH, \tensorh)$ of higher steps.

As usual, we require $\calH$ bracket-generating. Assume also, for the
sake of simplicity, that $\calH$ is equiregular, i.e.  there exists a flag
of subbundles $\calH= \calH^1 \subseteq \calH^2 \subseteq \calH^3
\subseteq \cdots \subseteq \calH^r$ such that
$$\calH^{k+1}_x = \spn \left\{ Z|_x, [A,Z] |_x \, \colon \, Z \in \Gamma(\calH^{k}),\ A \in \Gamma(\calH) \right\}, \quad x \in M.$$
Choose a metric tensor $\tensorv$ on $\calV$ and let $\tensorg =
\pr_{\calH}^* \tensorh + \pr_{\calV}^* \tensorv$ be the corresponding
Riemannian metric. Let $\calV_k$ be the orthogonal complement of
$\calH^k$ in $\calH^{k+1}$. Let $\pr_{\calV_k}$ be the projection to
$\calV_k$ relative to the splitting $\calH \oplus \calV_1 \oplus
\cdots \oplus \calV_{r-1}$. Define $\tensorv_k = \tensorv|_{\calV_k}$
and let $\tensorv^*_k$ be the corresponding co-metric. We could
attempt to construct a curvature-dimension inequality with
$\sfGamma^{\tensorh^*}(f),\, \sfGamma^{\tensorv_1^*}(f), \dots,
\sfGamma^{\tensorv_{r-1}^*}(f)$. However, a condition similar to
\eqref{CondB} could never hold in this case,
i.e. $\sfGamma^{\tensorh^*}(f, \sfGamma^{\tensorv_k^*}(f)) =
\sfGamma^{\tensorv^*_k}(f, \sfGamma^{\tensorh^*}(f))$ cannot hold for
any $k \leq r-2.$

To see this let $\alpha$ and $\beta$ be forms that only are
non-vanishing on respectively $\calV_{k}$ and $\calV_{k+1}$ for $k
\leq r-2$. Then $\sfGamma^{\tensorh^*}(f, \sfGamma^{\tensorv_k^*}(f))
= \sfGamma^{\tensorv^*_k}(f, \sfGamma^{\tensorh^*}(f))$ holds if and
only if $\rnabla \tensorh^* = 0$ and $\rnabla_A \tensorv_k^* =0$ for
any $A \in \Gamma(\calH)$. Hence we obtain
$$0 = (\rnabla_{A} \tensorv_k^*)(\alpha, \beta) = \beta([A, \shvk \alpha]).$$
However, this is a contradiction, since by our construction,
$\calV_{k+1}$ must be spanned by orthogonal projections of brackets on
the form $[A, Z ]$, $A \in \Gamma(\calH)$, $Z \in \Gamma(\calV_{k})$.

\appendix
\section{Graded analysis on forms}  \label{sec:Appendic}
\subsection{Graded analysis on forms} \label{sec:GradedForms} Let $(M,
\calH, \tensorh)$ be a sub-Riemannian manifold with an integrable
complement $\calV$, and let $\tensorv$ be a chosen positive definite
metric tensor on $\calV$. Let $\tensorg = \pr_{\calH}^*
\tensorh + \pr_{\calV}^* \tensorv$ be the corresponding Riemannian
metric. The subbundle $\calV$ gives us a foliation of $M$, and
corresponding to this foliation we have a grading on forms, see
e.g. \cite{AlTo91,Alv92}. Let $\Omega(M)$ be the algebra of
differential forms on $M$. Let $\Ann(\calH)$ and $\Ann(\calV)$ be the
subbundles of $T^*M$ of elements vanishing on respectively $\calH$ and
$\calV$. If either $a$ or $b$ is a negative integer, then $\eta \in
\Omega(M)$ is a homogeneous element of degree $(a,b)$ if and only if
$\eta = 0$. Otherwise, for nonnegative integers $a$ and $b$, 
$\eta$ is a homogeneous element of degree $(a,b)$, if it is a sum of
elements which can be written as
$$\alpha \wedge \beta , \quad \textstyle \alpha \in \Gamma(\bigwedge^a \Ann(\calV)),\ \beta \in \Gamma(\bigwedge^b \Ann(\calH)).$$

Relative to this grading, we can split the exterior differential $d$
into graded com\-po\-nents
$$d = d^{1,0} + d^{0,1} + d^{2,-1} .$$
The same is true for its formal dual
$$\delta = \delta^{-1,0} + \delta^{0,-1} + \delta^{-2,1} ,$$
i.e. the dual with respect to the inner product on forms of compact
support $\alpha, \beta$, defined by
\begin{equation} \label{HodgeProduct} \langle \alpha, \beta \rangle =
  \int_M \alpha \wedge \star \beta , \qquad \alpha,\beta \text{ of compact support}, \end{equation} where $\star$ is
the Hodge star operator defined relative to $\tensorg$.  Note that
$\delta^{-a,-b}$ is the formal dual of $d^{a,b}$ from our assumptions
that $\calH$ and $\calV$ are orthogonal. We will give formulas for
each graded component.

\subsection{Metric-preserving complement and local representation}
We will use $\flat:TM \to T^*M$ for the map $v \mapsto \tensorg(v, \newbullet)$ with inverse $\sharp$. Let
$\rnabla$ be defined as in \eqref{rnabla} relative to $\tensorg$ and
the splitting $TM = \calH \oplus \calV$.  If $\alpha$ is a one-form
and $A_1, \dots, A_n$ and $V_1, \dots, V_\nu$ are respective local
orthonormal bases of $\calH$ and $\calV$, then locally
\begin{align*} d\alpha = & \sum_{i=1}^n \flat A_i \wedge \rnabla_{A_i}
  \alpha + \sum_{s=1}^\nu \flat V_s \wedge \rnabla_{V_s} \alpha -
  \alpha \circ \calR ,
\end{align*}
and hence each of the three terms are local representations of
respectively $d^{1,0} \alpha,$ $d^{0,1} \alpha$ and $d^{2,-1}
\alpha$. Local representations on forms of all orders follow.

Assume now that $\rnabla \tensorg = 0$, i.e. $\calV$ is a
metric-preserving compliment of $(\calH, \tensorh)$ with a metric
tensor $\tensorv$ satisfying $\rnabla \tensorv^* = 0$. From the
formula of $d^{1,0} \eta = \sum_{i=1}^n \flat A_i \wedge \rnabla_{A_i}
\eta$, we obtain $\delta^{-1,0}\eta = - \sum_{i=1}^n \iota_{A_i}
\rnabla_{A_i} \eta$ for any form $\eta$. Let $\srL$ be the
sub-Laplacian of $\calV$ or equivalently $\vol$. Let $\Delta$ be the
Laplacian of $\tensorg$.  Then it is clear that for any $f \in
C^\infty(M)$, we have
$$\Delta f = -\delta df, \qquad \srL f = -\delta^{-1,0} d^{1,0} f.$$
\begin{lemma} \label{lemma:srLDeltaCommute} For any form $\eta \in
  \Omega(M)$, we have
  \begin{equation} \label{DDeltaCommute} \delta^{-1,0} d^{0,1} \alpha
    = - d^{0,1} \delta^{-1,0} \alpha, \quad \delta^{0,-1} d^{1,0}
    \alpha = - d^{1,0} \delta^{0,-1} \alpha.\end{equation} As a
  consequence, for any $f \in C^\infty(M)$ we have $\srL \Delta f =
  \Delta \srL f.$
\end{lemma}

The following result is helpful for our computation in Part~I,
Lemma~3.3~(b) and Corollary~3.11.
\begin{lemma} \label{lemma:fromPartI}
  \begin{enumerate}[\rm(a)]
  \item For any horizontal $A \in \Gamma(\calH)$, a vertical $V
    \in \Gamma(\calV)$ and arbitrary vector filed $Z \in \Gamma(TM)$, we have
$$\tensorg(R^{\rnabla}(A,V) Z, A) = 0.$$
\item If $\rnabla \tensorg = 0$, then for every point $x_0$, there
  exist local orthonormal bases $A_1, \dots, A_n$ and $V_1, \dots
  V_\nu$, defined in a neighborhood of $x_0$, such that for any $Y \in
  \Gamma(TM)$,
$$\rnabla_Z A_i |_{x_0} = \frac{1}{2} \sharp \tensorg(Z, \calR(A_i, \newbullet))|_{x_0}, \qquad \rnabla_Z V_s |_{x_0} = 0.$$
\end{enumerate}
\end{lemma}

\begin{proof}[of Lemma~\textup{\ref{lemma:srLDeltaCommute}}]
  It is sufficient to show one of the identities in
  \eqref{DDeltaCommute}, since $\delta^{-1,0}d^{0,1}$ is the formal
  dual of $\delta^{0,-1} d^{1,0}$. From
  Lemma~\ref{lemma:fromPartI}~(a), any $A \in
  \Gamma(\calH)$ and $V \in \Gamma(\calV)$ satisfy
$$\iota_A \rnabla_V \rnabla_A \alpha = \iota_A \rnabla_A \rnabla_V \alpha  + \iota_A \rnabla_{[V,A]} \alpha.$$
From the definition of $\rnabla$, it follows that
$T^{\rnabla}(A,V) = 0$, where $T^{\rnabla}$ is the torsion of
$\rnabla$. For a given point $x_0 \in M$, let $A_1, \dots, A_n$ and
$V_1, \dots, V_\nu$ be as in Lemma~\ref{lemma:fromPartI}~(b). All terms
below are evaluated at the point~$x_0$, giving us
\begin{align*}
  d^{0,1} \delta^{-1,0} \alpha &= -\sum_{s=1}^\nu \sum_{i=1}^n \flat V_s \wedge \rnabla_{V_s} \iota_{A_i} \rnabla_{A_i} \alpha \\
  &=  -\sum_{s=1}^\nu \sum_{i=1}^n  \flat V_s \wedge \iota_{\rnabla_{V_s} A_i} \rnabla_{A_i} \alpha -\sum_{s=1}^\nu \sum_{i=1}^n \flat V_s \wedge \iota_{A_i} \rnabla_{V_s} \rnabla_{A_i} \alpha \\
  &= -\frac{1}{2} \sum_{s=1}^\nu \sum_{i,j=1}^n \tensorg(V_s, \calR(A_i, A_j))  \flat V_s \wedge \iota_{A_j} \rnabla_{A_i} \alpha \\
  &\quad -\sum_{s=1}^\nu \sum_{i=1}^n \flat V_s \wedge \iota_{A_i} \rnabla_{A_i} \rnabla_{V_s} \alpha -\sum_{s=1}^\nu \sum_{i=1}^n \flat V_s \wedge \iota_{A_i} \rnabla_{\rnabla_{V_s} A_i - \rnabla_{A_i} V_s} \alpha \\
  &=   \sum_{s=1}^\nu \sum_{i,j=1}^n \iota_{A_i} \left(\flat V_s \wedge  \rnabla_{A_i} \rnabla_{V_s} \right) \alpha \\
  &= \sum_{s=1}^\nu \sum_{i,j=1}^n \iota_{A_i} \rnabla_{A_i}
  \left(\flat V_s \wedge \rnabla_{V_s} \right) \alpha = -\delta^{-1,0}
  d^{1,0} \alpha. \end{align*}

Next, we prove the identity $[\srL, \Delta] f = 0$. If we consider the degree $(1,1)$-part of $d^2 = 0$, we get
$$d^{0,1} d^{1,0} + d^{1,0} d^{0,1}  = 0.$$
The same relation will then hold for their formal duals.  Since
$\Delta f = \srL f - \delta^{0,1} d^{0,1} f,$ it is sufficient to show
that $\srL \delta^{0,-1} d^{0,1} f = \delta^{0,-1} d^{0,1} \srL f.$
This gives us the result
\begin{align*} & \srL \delta^{0,-1} d^{0,1} f = -\delta^{-1,0} d^{1,0}
  \delta^{0,-1} d^{0,1} f = - \delta^{0,-1} d^{0,1} \delta^{-1,0}
  d^{1,0} f = \delta^{0,-1} d^{0,1} \srL f,
\end{align*}
since we have to do an even number of permutations.\qed
\end{proof}

\subsection{Spectral theory of the sub-Laplacian} \label{sec:Spectral}
Let $L$ be a self-adjoint operator on $L^2(M,\vol)$ with domain
$\Dom(L)$. Define $\|f\|_{\Dom(L)}^2 = \|f\|^2_{L^2}
+\|Lf\|_{L^2}^2$. Write the spectral decomposition of $L$ as $L =
\int_{-\infty}^\infty \lambda dE_\lambda$ with respect to the
corresponding projector valued spectral measure $E_\lambda$. For any Borel
measurable function $\varphi\colon \real \to \real$, we write
$\varphi(L)$ for the operator $\varphi(L) := \int_{-\infty}^\infty
\varphi(\lambda) dE_{\lambda}$ which is self adjoint on its domain
$$\Dom(\varphi(L)) = \left\{ f \in L^2(M, \vol) \, \colon \, \int_{-\infty}^\infty \varphi(\lambda)^2 d\langle E_\lambda f, f \rangle \right\}.$$
In particular, if $\varphi$ is bounded, $\varphi(L)$ is defined on the
entire of $L^2(M,\vol)$. See \cite[Ch VIII.3]{ReSi80} for details.

Let $(M, \calH, \tensorh)$ be a sub-Riemannian manifold with
sub-Laplacian $\srL$ defined relative to a volume form $\vol$. Assume
that $\calH$ is bracket-generating and that $(M, \metricd_{cc})$ is
complete metric space, where~$\metricd_{cc}$ is the
Carnot-Carath\'eodory metric of~$(\calH, \tensorh)$. Then
$$\int_M f \srL g \, \dvol = \int_M g \srL f \, \dvol\quad\text{and}\quad\int f \srL
f \dvol \leq 0.$$ From \cite[Section 12]{Str86}, we have that $\srL$ is
a an essentially self adjoint operator on $C_c^\infty(M)$. We denote
its unique self-adjoint extension by $\srL$ as well with domain
$\Dom(\srL) \subseteq L^2(M, \vol).$

Since $\srL$ is non-positive and the maps
$\lambda \mapsto e^{t\lambda/2}$ and $\lambda \mapsto \lambda^j
e^{t\lambda/2}$ are bounded on 
$(- \infty, 0]$ for $t > 0$, $j > 0$, we have that $f \mapsto e^{t/2 \srL} f$ is
a map from $L^2(M, \vol)$ into $\cap_{j=1}^\infty \Dom(\srL^j)$.
Define $P_tf$ as in Section~\ref{sec:GeneralOperator} with respect to $\frac{1}{2}\srL$-diffusions
for bounded measurable functions $f$. Then clearly $\| P_t f\|_{L^\infty} \leq \| f\|_{L^\infty}$. Since $\srL$ is symmetric with respect to $\vol$ and $P_t 1 \leq 1$, we obtain $\|P_t f\|_{L^1} \leq \| f\|_{L^1}$ as well.
The Riesz-Thorin theorem then ensures that $\|P_t f\|_{L^p} \leq \|f\|_{L^p}$ for any $1 \leq p \leq
\infty.$ In particular, $P_t f$ is in $L^2(M, \vol)$ whenever $f$ is in $L^2(M, \vol)$. This implies that $P_t f = e^{t/2\srL} f$ for any bounded $f \in L^2(M, \vol)$ by the following result.  
\begin{lemma}[\rm{\cite[Prop]{Li84}}, {\cite[Prop 4.1]{BaGa12}}]  \label{lemma:Lpunique} 
Let $L$ be equal to the Laplacian $\Delta$ or sub-Laplacian $\srL$ defined relative to a complete Riemannian or sub-Riemannian metric, respectively. Let $u_t(x)$ be a solution in $L^2(M,\vol)$ of the heat equation
$$(\partial_t - L) u_t = 0, \quad u_0 = f,$$
for a function $f \in L^2(M,\vol)$. Then $u_t(x)$ is the unique solution of this equation in $L^2(M,\vol)$.
\end{lemma}

Hence, we will from now on just write $P_t = e^{t/2 \srL}$ without much abuse of notation.

\subsubsection{Global bounds using spectral theory}
We now introduce some additional assumptions. Assume that $\tensorg$
is a complete Riemannian metric with volume form $\vol$, such that
$\tensorg|_{\calH} = \tensorh, \calH^\perp = \calV$ and
$\tensorg|_{\calV} = \tensorv.$ Let $\Delta$ be the Laplace-Beltrami
operator of $\tensorg$ and write $\Delta f = \srL f + \verticalL f$
where $\verticalL f = \dv \shv df$. Since $\tensorg$ is complete,
$\Delta$ is also essentially self-adjoint on $C_c^\infty(M)$ by
\cite{Str83} and we will also denote its unique self-adjoint extension
by the same symbol.

Assume that $\rnabla \tensorg = 0$ where $\rnabla$ is defined as in
\eqref{rnabla}. Recall that $\srL$ and $\Delta$ commute on
$C_c^\infty(M)$ by Lemma~\ref{lemma:srLDeltaCommute}.

\begin{lemma} \label{lemma:SpecC}\quad
  \begin{enumerate}[\rm(a)]
  \item The operators $\srL$ and $\Delta$ spectrally commute, i.e. for
    any bounded Borel function $\varphi: \real \to \real$ and $f \in
    L^2(M,\vol)$,
$$\varphi(\srL) \varphi(\Delta)f = \varphi(\Delta) \varphi(\srL)f.$$
Also $\Dom(\Delta) \subseteq \Dom(\srL).$
\item Assume that $\srL$ satisfies the assumptions of
  Theorem~\textup{\ref{th:CD}} with $\mcalR >0$. Then there exist a
  constant $C = C(\rho_1, \rho_2)$ such that for any $f \in
  C^\infty(M) \cap \Dom(\srL^2)$,
$$C\|f\|_{\Dom(\srL^2)}^2 = C\left(\|f\|^2_{L^2} + \|\srL^2 f\|_{L^2}^2 \right),$$
is an upper bound for
$$\int_M \sfGamma^{\tensorh^*}(f) \dvol, \quad \int_M \sfGamma^{\tensorh^*}_2(f) \dvol, \quad \int_M \sfGamma^{\tensorv^*}(f) \dvol \quad \text{and} \quad \int_M \sfGamma^{\tensorv^*}_2(f) \dvol.$$
\end{enumerate}
\end{lemma}

\begin{proof}
  \begin{enumerate}[\rm(a)]
  \item Note first that for any $f \in C_c^\infty(M)$, using
    Lemma~\ref{lemma:srLDeltaCommute} and the inner product
    \eqref{HodgeProduct}
    \begin{align*}
      \int_M \verticalL f \srL f \dvol  
     &=  \langle \delta^{0,-1} d^{0,1} f, \delta^{-1,0} d^{1,0} f\rangle 
    = \langle d^{0,1} f, d^{0,1} \delta^{-1,0} d^{1,0}f \rangle \\
      &= - \langle d^{0,1} f, \delta^{-1,0} d^{0,1} d^{1,0}f \rangle =
      \langle d^{1,0} d^{0,1} f, d^{1,0} d^{0,1} f \rangle \geq 0.
    \end{align*}
    Hence
$$\int_M (\srL f)^2 \dvol \leq \int_M ((\srL + \verticalL) f)^2 \dvol = \int_M (\Delta f)^2 \dvol,$$
and hence $\| \srL f\|_{L^2} \leq \|\Delta f\|_{L^2}$ is true for any
$f \in \Dom(\Delta)$. We conclude that $\Dom(\Delta) \subseteq
\Dom(\srL)$. Define $Q_t = e^{t/2 \Delta}$. It follows that, for any
$f \in \Dom(\srL)$, $u_t = \srL Q_t f$ is an $L^2(M,\vol)$ solution of
$$\left(\frac{\partial}{\partial t} - \frac{1}{2} \Delta\right) u_t = 0, \quad u_0 =
\srL f.$$ In conclusion, by Lemma~\ref{lemma:Lpunique} we obtain $\srL
Q_t f = Q_t \srL f$.

For any $s >0$ and $f \in L^2(M, \vol)$, we know that $Q_sf \in
\Dom(\Delta) \subseteq \Dom(\srL)$, and since $$(\partial_t -
\frac{1}{2} \srL) Q_s P_t f = 0,$$ it again follows from
Lemma~\ref{lemma:Lpunique} that $P_t Q_s f = Q_s P_t f$ for any $s,t
\geq 0$ and $f \in L^2(M, \vol)$. It follows that the operators
spectrally commute, see \cite[Chapter VIII.5]{ReSi80}.

\item From Theorem~\ref{th:CD}, we know that $\srL$ satisfies \eqref{CD}
  with $\rho_2 >0$ and an appropriately chosen value of~$c$. The proof
  is otherwise identical to \cite[Lemma~3.4 \& Prop~3.6]{BaGa12} and
  is therefore omitted.\qed
\end{enumerate}
\end{proof}

\subsubsection{Proof of Theorem~\textup{\ref{th:CondARiemann}}} \label{sec:ProofCondARiemann}
We are going to prove that \eqref{CondA} holds without using
stochastic analysis. We therefore need the following lemma.
\begin{lemma}[\rm{\cite[Prop 4.2]{BaGa12}}] \label{lemma:preSubM}  
Assume
  that $(M, \tensorg)$ is a complete Riemannian manifold. For any $T> 0$, let $u,
  v\in C^\infty(M \times [0,T])$, $(x,t) \mapsto u_t(x)$, $(x,t)
  \mapsto v_t(x)$ be smooth functions satisfying the following
  conditions:
  \begin{enumerate}[\rm(i)]
  \item For any $t \in [0,T]$, $u_t \in L^2(M,\vol)$ and $\int_0^T
    \|u_t\|_{L^2} dt < \infty.$
  \item For some $1\leq p \leq \infty,$ $\int_0^T
    \|\sfGamma^{\tensorh^*}(u_t)^{1/2} \|_{L^p} \dvol < \infty.$
  \item For any $t \in [0,T]$, $v_t \in L^q(M,\vol)$ and $\int_0^T
    \|v_t\|_{L^q} dt < \infty$ for some $1 \leq q \leq \infty$.
  \end{enumerate}
  Then, if $(L + \frac{\partial}{\partial t} ) u \geq v$ holds on $M
  \times [0,T]$, we have
$$P_T u_T \geq u_0 + \int_0^t P_t v_t dt.$$
\end{lemma}

  Let $P_t = e^{t/2 \srL}$.  For given compactly supported $f \in C_c^\infty(M)$ and $T
  >0$, define function
  \begin{equation} \label{zte} z_{t,\ve} =
    \left(\sfGamma^{\tensorv^*}(P_{T-t} f) + \ve^2 \right)^{1/2} -
    \ve,\end{equation} with $\ve > 0, t\in [0,T]$. Since $P_t f \in
  \Dom(\srL^2)$, Lemma~\ref{lemma:SpecC}~(b) tells us that,
$$\|z_{t,\ve}\|_{L^2} \leq \int_M \sfGamma^{\tensorv^*}(P_{T-t}) \vol \leq C \| P_{T-t} f\|_{\Dom(\srL^2)} < \infty,$$
so that $z_{t,\ve} \in L^2(M, \vol)$. By
Proposition~\ref{prop:DoubleGamma},
\begin{equation} \label{vdoubleGamma} \sfGamma^{\tensorh^*}(z_{t,\ve})
  \leq \frac{\sfGamma^{\tensorh^*}(\sfGamma^{\tensorv^*}(P_{T-t}
    f))}{4 z_{t,\ve} + \ve} \leq \sfGamma_2^{\tensorv^*}(P_{T-t} f)
  .\end{equation} From Lemma~\ref{lemma:SpecC}~(b) it follows that
both (i) and (ii) of Lemma~\ref{lemma:preSubM} is satisfied. Hence, using that from Proposition~\ref{prop:DoubleGamma}
\begin{align} \label{ztesubmartingale} & \left(\partial_t -
    \frac{1}{2} \srL \right) z_{t,\ve} \\ \nonumber &\quad =
  \frac{1}{2 (z_{t,\ve} + \ve)^3} \left( \sfGamma^{\tensorv^*}(P_{T-t}
    f) \sfGamma^{\tensorv^*}_2(P_{T-t} f) - \frac{1}{4}
    \sfGamma^{\tensorh^*}(\sfGamma^{\tensorv^*}(P_{T-t} f)) \right)
  \geq 0,
\end{align}
we get $P_T z_{T,\ve} = P_T (\sfGamma^{\tensorv^*}(f) + \ve^2 )^{1/2}
- P_T \ve \geq z_{0,\ve} = (\sfGamma^{\tensorv^*}(P_Tf) + \ve
^2)^{1/2} - \ve$. By letting $\ve $ tend to $0$, we obtain
\begin{equation} \label{vgradientb} \sqrt{\sfGamma^{\tensorv^*}(P_Tf)}
  \leq P_T \sqrt{\sfGamma^{\tensorv^*}(f)}.\end{equation} Next, let
$y_{t,\ve} = \left(\sfGamma^{\tensorh^*}(P_{T-t} f) + \ve^2
\right)^{1/2} - \ve$, choose any $\alpha > \max\{ -\rRicH, \MRicHV^2
\} \geq 0$ and define
$$u_{t,\ve} = e^{-\alpha/2(T-t)} \big(y_{t,\ve} + \ell \sfGamma^{\tensorv^*}(P_{T-t}(f)) \big) .$$
Note first that
\begin{align*}
  &\left(\frac{\partial}{\partial t} + \frac{1}{2} \srL \right) u_{t,\ve}\\&\quad = \frac{e^{-\alpha/2(T-t)}}{2 y_{t,\ve} + 2 \ve} \left( \sfGamma^{\tensorh^*}_2(P_{T-t} f) + \ell y_{t,\ve} \sfGamma^{\tensorv^*}_2(P_{T-t} f) - \frac{1}{4y_{t,\ve}^2} \sfGamma^{\tensorh^*}(\sfGamma^{\tensorh^*}(P_{T-t}f)) \right) \\
  &\qquad + \frac{\alpha e^{-\alpha/2(T-t)}}{2 y_{t,\ve} + 2 \ve}
  \left( \sfGamma^{\tensorh^*}(P_{T-t} f) + \ve+ \ell y_{t,s}
    \sfGamma^{\tensorv^*}(P_{T-t}f) \right) .
\end{align*}
We use Proposition~\ref{prop:DoubleGamma} with $\ell$ replaced by
$\ell y_{t,\ve}$ to get
\begin{align*}
\frac{1}{4 y_{t,\ve}^2} \sfGamma^{\tensorh^*}(\sfGamma^{\tensorh^*}(P_{T-t}f)) 
\leq \sfGamma^{\tensorh^*}_2(f) &- (\rRicH - c^{-1} - \ell^{-1} y_{t,\ve}^{-1}) \sfGamma^{\tensorh^*}(f) \\
&+ \ell y_{t,\ve} \sfGamma^{\tensorv^*}_2(f) - c \MRicHV^2 \sfGamma^{\tensorv^*}(f).
\end{align*}
As a result, for any $c > 0$, $\left(\frac{\partial}{\partial t} + \frac{1}{2} \srL \right) u_{t,\ve}$ has lower bound
\begin{align*}
  &\frac{e^{-\alpha/2(T-t)}}{2 y_{t,\ve} + 2\ve} \left( (\rRicH - c^{-1} - \ell^{-1} y_{t,\ve}^{-1} ) \sfGamma^{\tensorh^*}(P_{T-t} f) - c \MRicHV^2 \sfGamma^{\tensorv^*}(P_{T-t} f) \right) \\
  &\qquad + \frac{\alpha e^{-\alpha/2(T-t)}}{2 y_{t,\ve} + 2\ve} \left( \sfGamma^{\tensorh^*}(P_{T-t} f) + \ell y_{t,\ve} \sfGamma^{\tensorv^*}(P_{T-t}f) \right). 
\end{align*}
Since it is true for any value of $c >0$, it remains true for $c =
\ell y_{t,\ve}$, and hence
\begin{align*}
  \left(\frac{\partial}{\partial t} + \frac{1}{2} \srL \right)
  u_{t,\ve} \geq & - \frac{e^{-\alpha/2(T-t)}}{\ell} .
\end{align*}
In a similar way as before, we can verify that the conditions of
Lemma~\ref{lemma:preSubM} hold by using Lemma~\ref{lemma:SpecC}. We
can hence conclude that
\begin{align*}
u_{0,\ve} &=  e^{-\alpha T/2} (y_{0,\ve} + \ell \sfGamma^{\tensorv^*}(P_Tf)) \\
  & \leq P_T u_{T,\ve} + \int_0^T P_t \frac{e^{-\alpha(T-t)/2}}{\ell} \, dt \\
  & \leq P_T \left( y_{T,\ve} + \ell \sfGamma^{\tensorv^*}(f) \right)
  + \frac{2}{\alpha \ell} \left(1- e^{- \alpha T/2} \right).
\end{align*}
Multiplying with $e^{\alpha T/2}$ on both sides, letting $\ve \to 0$
and $\alpha \to k := \max\{ - \rRicH, \McalR\}$, we finally get that for any $\ell > 0$,
\begin{align} \label{hgradientb} & \sqrt{\sfGamma^{\tensorh^*}(P_Tf)}
  + \ell \sfGamma^{\tensorv^*}(P_T f) \leq e^{k T/2} P_T\left(
    \sqrt{\sfGamma^{\tensorh^*}(f)} + \ell \sfGamma^{\tensorv^*}(f)
  \right) + \ell^{-1} F_k(T),\end{align}
where 
\begin{equation*}F_k(t) = \begin{cases} \frac{2}{k} ( e^{k t/2} - 1) & \text{if } k >0, \\
    t & \text{if } k =0. \end{cases}  
\end{equation*}
Since this estimate holds pointwise, it holds for 
$\ell = \big(P_T\sfGamma^{\tensorv^*}(f)- \sfGamma^{\tensorv^*}(P_Tf)\big)^{-1/2}$ 
or $\ell = \infty$ at points where $P_T \sfGamma^{\tensorv^*}(f)- \sfGamma^{\tensorv^*}(P_Tf) = 0$. The resulting inequality is
\begin{align} \label{chgradientb} &  \sqrt{\sfGamma^{\tensorh^*}(P_Tf)} \leq  e^{k T/2} P_T \sqrt{\sfGamma^{\tensorh^*}(f)} + (e^{k T/2} +F_k(T)) \sqrt{P_T\sfGamma^{\tensorv^*}(f)- \sfGamma^{\tensorv^*}(P_tf)} .\end{align}

We will now show how this inequality implies \eqref{CondA}. Since $\tensorg$ is complete, there exist a sequence of compactly
supported functions $g_n \in C^\infty(M)$ satisfying $g_n \uparrow 1$ pointwise and $\|\sfGamma^{\tensorh^* + \tensorv^*}(g_n)\|_{L^\infty} \to 0.$ It follows from equation \eqref{vgradientb} and \eqref{chgradientb} that
$$\lim_{n \to \infty} \|\sfGamma^{\tensorh^* + \tensorv^*}(P_t g_n)\|_{L^\infty} \to 0$$
as well. Hence, since $P_t g_n \to P_t 1$ and $\| dP_t g_n\|_{\tensorg^*}$ approach $0$
uniformly, we have that $\sfGamma^{\tensorg^*}(P_t1) = 0$. It
follows that $P_t1 = 1$.

To finish the proof, consider a smooth function $f \in C^\infty(M)$ with $\|f\|_{L^\infty} < \infty$ and $\|\sfGamma^{\tensorh^* + \tensorv^*}(f)\|_{L^\infty} < \infty.$ Define $f_n = g_n f \in C^\infty(M)$. Then $P_Tf_n \to P_T f$ pointwise. It follows that
\begin{equation} \label{DCT} \int_a^b dP_T f(\dot \gamma(t)) \, dt 
= \lim_{n\to \infty} \int_a^b dP_T f_n \left(\dot \gamma(t)\right) \, dt \end{equation}
for any smooth curve $\gamma:[a,b] \to M$. We want to use the dominated convergence theorem to show that the integral sign and limit on the right side of \eqref{DCT} can be interchanged.

Without loss of generality, we may assume that $\|\sfGamma^{\tensorh^*}(g_n)\|_{L^\infty} < 1$ for any $n$. We then note that
$$\left\|\sqrt{\sfGamma^{\tensorh^* + \tensorv^*}(f_n)} \right\|_{L^\infty} \leq \|f\|_{L^\infty} + \left\| \sqrt{\sfGamma^{\tensorh^*}(f)} \right\|_{L^\infty}=:K < \infty.$$
This relation, combined with \eqref{vgradientb} and \eqref{chgradientb}, gives us
$$\|\sqrt{ \sfGamma^{\tensorh^*+ \tensorv^*} (P_Tf_n) }\|_{L^\infty} \leq \left( 2e^{kT/2} + F_k(T) +1 \right)K.$$
Furthermore, the dominated convergence theorem tells us that both $P_T \sfGamma^{\tensorv^*}(f_m-f_n)$ and 
$\lim_{n \to \infty} P_T \sfGamma^{\tensorh^*}(f_n - f_m) $ approach $0$ pointwise as $n,m \to \infty$. By inserting $f_n-f_m$ into \eqref{vgradientb} and \eqref{chgradientb}, we see that $\sfGamma^{\tensorh^* + \tensorv^*}(P_T f_n)$ at any fixed point is a Cauchy sequence and hence convergent. We conclude that
$$\int_a^b dP_T f (\dot \gamma(t)) dt = \int_a^b \left(\lim_{n \to \infty} dP_T f_n\right)(\dot \gamma(t)) dt.$$
It follows that $dP_t f - \lim_{n \to \infty} dP_t f$ vanishes outside a set of measure zero along any curve, so $$\|\sfGamma^{\tensorh^* + \tensorv^*}(P_T f) \|_{L^\infty} = \lim_{n \to \infty} \|\sfGamma^{\tensorh^* + \tensorv^*}(P_T f_n ) \|_{L^\infty} < \infty.$$
In conclusion, we have proven that condition \eqref{CondA} holds. Without any loss of
generality we can put $\ell = 1$, since we can obtain all the other
inequalities by replacing $f$ with $\ell f$.
\qed

\begin{remark}
  If we know that any $\frac{1}{2} \srL$-diffusion starting at a point
  has infinite lifetime then using Lemma \ref{lemma:SpecC}, we can
  actually make a probabilistic proof. We outline the proof
  here. We will only prove the inequality \eqref{vgradientb} as the
  proof of \eqref{hgradientb} is similar.

  We will again use $z_{t,\ve}$ as in \eqref{zte}. Let $ X = 
    X(x)$ be an $\frac{1}{2} \srL$-diffusion with $X_0(x) = x
  \in M$. We define $Z^{\ve}$ by $Z^{\ve}_t = z_{t,\ve}
  \circ X_t$.  Then $Z^{\ve}$ is a local submartingale by
  \eqref{ztesubmartingale}.  By using the Burkholder-Davis-Gundy
  inequality, there exist a constant $B$ such that
  \begin{align*} \expect\left[\sup_{0 \leq s \leq t} Z_s^{\ve}
    \right] \leq B \expect\left[\sqrt{ \langle Z^{\ve}
        \rangle_{t } } \right] + z_{0,\ve}(x) + \expect\left[
      \int_0^{t} (\partial_s - \frac{1}{2} \srL) z_{s,\ve} \circ X_s \,ds \right] \end{align*}
  where $\langle Z^{\ve}
  \rangle_t = \int_0^t \sfGamma^{\tensorh^*}(z_{s,\ve}) \circ X_s\, ds$ is the quadratic variation of $Z^{\ve}$. By
  the Cauchy-Schwartz inequality and the bound \eqref{vdoubleGamma},
  we get the conclusion
  \begin{align*}
    \expect\left[\sqrt{ \langle Z^{\ve} \rangle_t } \right] \leq
    p_{2t}(x,x)\, \sqrt{ \int_0^t \|\sfGamma^{\tensorv^*}_2(P_{T-s}
      f) \|_{L^1} \,dt} < \infty
  \end{align*}
  which means that $ \expect\left[\sup_{0 \leq s \leq t \wedge
      \tau} Z_s^{\ve} \right] <\infty$. Hence, $Z^{\ve}$ is a true submartingale, giving us \eqref{vgradientb}.
\end{remark}

\subsection{Interpretation of $\RicHV$} \label{sec:RicHV} Let $\calV$
be any integrable subbundle. Choose a subbundle $\calH$ such that $TM
= \calH \oplus \calV$.  Any such choice of $\calH$ correspond uniquely
to a constant rank endomorphism $\pr = \pr_{\calV}: TM \to \calV
\subseteq TM$. This can be considered as a splitting of the short
exact sequence $\calV \to TM \stackrel{F}{\to} TM/\calV.$

Let $\Omega(M)$ be the the exterior algebra of $M$ with $\integer
\times \integer$-grading of Section~\ref{sec:GradedForms}. Choose
nondegenerate metric tensors
$$\tensorv \in \Gamma(\Sym^2 \calV^*) \text{ and } \widecheck{\tensorg} \in \Gamma(\Sym^2 (TM/\calV)^*)$$ on $\calV$ and $TM/\calV$. Since $\bigwedge^\nu \calV^* \oplus \bigwedge^n (TM/\calV)^*$ is canonically isomorphic to  $\bigwedge^{n+\nu} T^*M$, the choices of $\tensorv$ and $\widecheck{\tensorg}$ gives us a volume form $\vol$ on $M$.

We also have an energy functional defined on projections to
$\calV$. Relative to $\pr$, define a Riemannian metric $\tensorg_{\pr}
= F^* \widecheck{\tensorg} + \pr^* \tensorv$. We introduce a
functional $E$ on the space of projections $\pr$ by
$$E(\pr) = \int_M \| \calR_{\pr}\|^2_{\wedge^2 \tensorg_{\pr}^* \otimes \tensorg_{\pr}} \dvol$$
where $\calR_{\pr}$ is the curvature of $\calH = \ker \pr$. We can
only be sure that the integral is finite if $M$ is compact, so we will
assume this, and consider our calculations as purely formal when this
is not the case.

Let $\nabla= \nabla^{\pr}$ be the restriction of the Levi-Civita
connection of $\tensorg_{\pr}$ to $\calV$.  Introduce a exterior
covariant derivative of $d_{\nabla}$ on $\calV$-valued forms in the
usual way, i.e. for any section $V \in \Gamma(\calV)$, we have
$d_{\nabla} V = \nabla_{\!\newbulletsub}^{\mathstrut}\! V$ and if $\alpha$ is a
$\calV$-valued $k$-form, while $\mu$ is a form in the usual sense,
then
$$d_\nabla (\alpha \wedge \mu) = (d_\nabla \alpha) \wedge \mu + (-1)^k \alpha \wedge d\mu.$$
We can split this operator into graded components $d_{\nabla} =
d^{1,0}_{\nabla} + d^{0,1}_{\nabla} + d^{2,-1}_{\nabla} $ and do the
same with its formal dual $\delta_{\nabla} = \delta^{-1,0}_\nabla +
\delta^{0,-1}_{\nabla} + \delta_{\nabla}^{2,-1}$.

\begin{proposition}
  The endomorphism $\pr$ is a critical value of $E$ if and only if
  $\delta^{-1,0}_{\nabla} \calR = 0$. In particular, if~$\tensorg$
  satisfies
  \begin{equation} \label{IIis0} \tr_{\calV} (\calL_{A} \tensorg)(\times, \times) = 0,
    \quad \text{for any } A \in \Gamma(\calH),\end{equation} then $\pr$ is a critical value if and
  only if $\RicHV = 0.$
\end{proposition}

Recall from Part I, Section~2.4 that condition \eqref{IIis0} is equivalent to the leafs of the foliation of $\calF$ being minimal submanifolds. If $\calV$ is the vertical bundle of a submersion $\pi\colon M \to B$,
then we can identify $TM/\calV$ with $\pi^* TB$. In this case, a critical value of $E$ can be considered as an optimal way
of choosing an Ehresmann connection on~$\pi$.

\begin{proof}
  We write $\id:= \id_{TM}$ for the identity on $TM$. Let $\pr$ be a
  projection to $\calV$ and $\alpha: TM \to \calV$ be any $\calV$-values one-from
  with $\calV \subseteq \ker \alpha$. Define a curve in the space
  projections $\pr_t = \pr + t\alpha$. Then
$$\tensorg_t(v,v) := \tensorg_{\pr_t}(v,v) = \tensorg_{\pr}(v,v) +2t \tensorv(\alpha v, \pr v) + t^2 \tensorv(\alpha v, \alpha v).$$
Let $\calR_t$ be the curvature of $\pr_t$. Then
\begin{align*} \calR_t(A,Z) &=   \calR(A,Z) + t \alpha [(\id -\pr) A, (\id - \pr)Z]  \\
  & \quad- t \left(\pr [\alpha A, (\id -\pr) Z ] + \pr [ (\id -\pr) A, \alpha Z
    ]\right) + O(t^2).\end{align*} If $\nabla^t = \nabla^{\pr_t}$, then
\begin{align*}\nabla^t_A V &= \nabla_A V + \frac{1}{2} t d_\nabla
  \alpha(A,V) - \frac{1}{2} t \shv \tensorg(d_\nabla \alpha(A,\newbullet), V),
  \quad\text{and}\\
  d_{\nabla^t} \pr_t &=d_{\nabla} \pr + \frac{1}{2} t (d_{\nabla}
  \alpha)_{1,1} - \frac{1}{2} t (d_{\nabla} \alpha )^\top_{1,1} + t d_{\nabla} \alpha +
  O(t^2),
\end{align*}
where $(d_\nabla \alpha)_{1,1}$ is the (1,1)-graded component of $d_\nabla
\alpha$ and $$\tensorv((d_{\nabla}\alpha)^\top_{1,1}(A,V_1),V_2) = \tensorv((d_\nabla
\alpha)_{1,1}(A,V_1),V_2).$$ Since $\calR_t = -d_{\nabla}^{2,-1} \pr_t$, we get
\begin{align*}
  \left. \frac{d}{dt} E(\pr_t) \right|_{t=0} &= \int_M (\wedge^2 \tensorg_{\pr}^* \otimes \tensorg_{\pr})( d_{\nabla}^{2,-1} \pr , d_\nabla^{1,0} \alpha) \dvol \\
  &= - \int_M (\tensorh^* \otimes \tensorv)(\delta_{\nabla}^{-1,0}
  \calR, \alpha) \dvol.
\end{align*}
Hence, $\pr$ is a critical value if and only if $\delta^{-1,0}_\nabla
\calR =0$.

We give a local expression for this identity. Let $A_1, \dots, A_n$ be
a local orthonormal basis of $\calH$. Then
\begin{align*}
  & \delta^{-1,0}_{\nabla} \calR = - \sum_{k=1}^n (\znabla_{A_k}
  \calR)(A_k, \newbullet) + \calR(N, \newbullet)
\end{align*}
where $N$ is defined by $\tensorg_{\pr}(A,N) = - \frac{1}{2} \tr_{\calV}
(\calL_{\pr_{\calH} A} \tensorg)(\times, \times)$ and $\znabla$ is the (0,0)-degree component of the
Levi-Civita connection, i.e.
$$\znabla_A Z = \pr_{\calH} \nabla_A \pr_{\calH} Z + \pr_{\calV} \nabla_A \pr_{\calV} Z.$$
This coincides with $\RicHV$ when \eqref{IIis0} holds.\qed
\end{proof}

\subsection{If $\calV$ is not integrable} \label{sec:NotIntegrable}
Let $(M, \tensorg)$ be a complete Riemannian manifold and let $\calH$ be a bracket-generating subbundle of $TM$
with orthogonal complement $\calV$. Define $\rnabla$ as in \eqref{rnabla} with respect to $\tensorg$ and the splitting $TM = \calH \oplus \calV$ and assume that $\rnabla \tensorg = 0$. Let $\srL$ be the sub-Laplacian defined relative to $\calV$ or equivalently to the volume form of $\tensorg$. Then it may happen that \eqref{CD} holds for $\srL$ even without assuming that $\calV$ is integrable. More precisely, we will need the condition
\begin{equation} \label{traceZero} \tr \overline{\calR}(v,\calR(v, \newbullet)) =0, \qquad v \in TM, \end{equation}
$$\calR(A,Z) = \pr_{\calV} [ \pr_{\calH} A, \pr_{\calH} Z], \quad \overline{\calR}(A,Z) = \pr_{\calH} [ \pr_{\calV} A, \pr_{\calV} Z], \quad A,Z \in \Gamma(TM).$$
We refer to $\calR$ and $\overline{\calR}$ as respectively the curvature and the co-curvature of $\calH$.

In Part~I, Section~3.8, we showed that Theorem~\ref{th:CD} and Proposition~\ref{prop:DoubleGamma} hold with the same definitions and with $\calV$ not integrable, as long as \eqref{traceZero} also holds. The same is true for Theorem~\ref{th:CondARiemann}. We give some brief details regarding this.

First of all, in Section~\ref{sec:GradedForms}, the exterior derivative $d$ now also has a part of degree $(-1,2)$, determined by
$$d^{-1,2} f =0, \quad d^{-1,2} \alpha = - \alpha \circ \overline{\calR}, \qquad f \in C^\infty(M),\ \alpha \in \Gamma(T^*M),$$
and hence, the co-differential has a degree $(1,-2)$-part. However, these do not have any significance for our calculations. More troubling is the fact that both Lemma~\ref{lemma:fromPartI}~(a) and the formula for $\rnabla_Z V_s|_{x_0}$ in Lemma~\ref{lemma:fromPartI}~(b) are false when $\calV$ is not integrable. However, \eqref{traceZero} ensures that $$\sum_{i=1}^n \tensorg(R^{\rnabla}(A_i, V) Z, A_i) = 0$$ for any orthonormal basis $A_1, \dots, A_n$ of $\calH$ and vertical vector field $V$, which is all we need for the proof of Lemma~\ref{lemma:srLDeltaCommute}. Furthermore the same proof is still holds even if now $\rnabla_Z V_s|_{x_0} = \frac{1}{2} \sharp \overline{\calR}(Z, \newbullet)|_{x_0}$ in Lemma~\ref{lemma:fromPartI}~(b), as the extra terms cancel out.

Once Lemma~\ref{lemma:srLDeltaCommute} holds, there is no problem with the rest of the proof of Theorem~\ref{th:CondARiemann}. See Part I, Section 4.6 for an example where this theorem holds.

\bibliographystyle{abbrv}  
\bibliography{Bibliography.bib}

\end{document}